\documentclass[reqno]{amsart}
\usepackage{graphicx}
\usepackage{amstext}
\usepackage{amssymb}
\usepackage{amsmath}
\usepackage{color}
\newtheorem{theorem}{Theorem}[section]
\newtheorem{lemma}[theorem]{Lemma}
\newtheorem{corollary}[theorem]{Corollary}
\newtheorem{proposition}[theorem]{Proposition}
\theoremstyle{definition}
\newtheorem{definition}[theorem]{Definition}

\theoremstyle{remark}
\newtheorem{remark}[theorem]{Remark}
\numberwithin{equation}{section}
\usepackage{hyperref}
\hypersetup{hidelinks=true,bookmarks=true}

\begin{document}
	
\title{Blow-up solutions for $L^2$-supercritical gKdV equations with exactly $k$ blow-up points}\def\rightmark{BLOW-UP SOLUTIONS FOR GKDV EQUATIONS WITH K BLOW-UP POINTS}
	
\author{Yang LAN}
	
\address{Institut des Hautes \'{E}tudes Scientifiques, Bures-sur-Yvette, France}
	
\email{yanglan@ihes.fr}
	

\begin{abstract}
In this paper we consider the slightly $L^2$-supercritical gKdV equations $\partial_t u+(u_{xx}+u|u|^{p-1})_x=0$, with the nonlinearity $5<p<5+\varepsilon$ and $0<\varepsilon\ll 1$. In the previous work of the author, we know that there exists a stable self-similar blow-up dynamics for slightly $L^2$-supercritical gKdV equations. Such solution can be viewed as solutions with single blow-up point. In this paper we will prove the existence of solutions with multiple blow-up points, and give a description of the formation of the singularity near the blow-up time.
\end{abstract}
	
\maketitle

\section{Introduction}	
\subsection{Setting of the problem}
We consider the following gKdV equations:
\begin{equation}\label{CP}\begin{cases}
\partial_t u +(u_{xx}+u|u|^{p-1})_x=0, \quad (t,x)\in[0,T)\times\mathbb{R},\\
u(0,x)=u_0(x)\in H^1(\mathbb{R}),
\end{cases}
\end{equation}
with $1\leq p<+\infty$.
	
From the result of C. E. Kenig, G. Ponce and L. Vega \cite{KPV} and N. Strunk \cite{Stru}, \eqref{CP} is locally well-posed in $H^1$ and thus for all $ u_0\in H^1$, there exists a maximal lifetime $0<T\leq+\infty$ and a unique solution $u(t,x)\in C([0,T), H^1(\mathbb{R}))$ to \eqref{CP}. Besides, we have the blow-up criterion: either $T=+\infty$ or $T<+\infty$ and $\lim_{t\rightarrow T}\|u_x(t)\|_{L^2}=+\infty$.
	
\eqref{CP} admits two conservation laws, i.e. the mass and energy:
\begin{gather*}
M(u(t))=\int |u(t,x)|^2 \,dx=M(u(0)),\\
E(u(t))=\frac{1}{2}\int |u(t,x)|^2 \,dx-\frac{1}{p+1}\int |u(t,x)|^{p+1} \,dx=E(u(0)).
\end{gather*}
	
{For all $ \lambda>0$, $u_{\lambda}(t,x)=\lambda^{\frac{2}{p-1}}u(\lambda^3t,\lambda x)$ is also a solution. Moreover, the $\dot{H}^{\sigma_c}$ norm of the initial data with the index:
\begin{equation}
\sigma_c=\frac{1}{2}-\frac{2}{p-1},
\end{equation}
is invariant under this scaling.}
	
We introduce the  ground state $\mathcal{Q}_p$, which is the unique radial nonnegative function with exponential decay at infinity to the following equation:
\begin{equation}
\mathcal{Q}_p''-\mathcal{Q}_p+\mathcal{Q}_p|\mathcal{Q}_p|^{p-1}=0.
\end{equation}
$\mathcal{Q}_p$ plays a distinguished role in the analysis. It provides a family of travelling wave solutions:
$$u(t,x)=\lambda^{\frac{2}{p-1}}\mathcal{Q}_p(\lambda(x-\lambda^2 t-x_0)),\quad (\lambda,x_0)\in\mathbb{R}_{+}^*\times\mathbb{R}.$$
	
For $p<5$ or equivalently $\sigma_c<0$, \eqref{CP} is called $L^2$ subcritical. The mass and energy conservation laws imply that the solution is always global and bounded in $H^1$.
	
{For $p=5$, the solution is called $L^2$ critical. From variation arguments \cite{W1}, we know that if $\|u_0\|_{L^2}<\|\mathcal{Q}_5\|_{L^2}$, then the solution to \eqref{CP} is always global and bounded in $H^1$. 
	
While for $\|u_0\|_{L^2}\geq\|\mathcal{Q}_5\|_{L^2}$, blow up may occurs. The blow up dynamics for solution with slightly supercritical mass:
\begin{equation}
\|\mathcal{Q}_5\|_{L^2}<\|u_0\|_{L^2}<\|\mathcal{Q}_5\|_{L^2}+\alpha^*
\end{equation}
has been developed in a series paper of Martel and Merle \cite{MM4,MM5,MM3,M1}. In particular, they obtain the existence of blow up solutions with negative energy, and the classification of the ground state $\mathcal{Q}_5$ as the unique global attractor for blow up solutions in $H^1$. 

In \cite{MMNR,MMR1,MMR2,MMR3}, Martel, Merle, Nakanishi and Rapha\"{e}l, give a comprehensive study of the asymptotic dynamics near the ground state: classification of the flow near soliton; existence of the minimal mass blow up solutions; exotic blow up regime; condimension $1$ threshold manifold for unstable regime.
    
\subsection{On the supercritical problems}Let us consider the focusing $L^2$ supercritical nonlinear Schr\"{o}dinger equations (NLS)
\begin{equation}\label{NLS}\tag{\rm NLS}
\begin{cases}
i\partial_t u +\Delta u+u|u|^{p-1}=0, \quad (t,x)\in[0,T)\times\mathbb{R}^d,\\
u(0,x)=u_0(x)\in H^1(\mathbb{R}^d),
\end{cases}
\end{equation}
with $1+\frac{4}{d}<p<+\infty$. 
The blow up dynamics for supercritical NLS is mostly open. Only until recently, a few special examples are known. From \cite{MRS1,R1,RS2}, there exist blow-up solutions with $\log$-$\log$ blow up rate for $p=5$ and $d\geq2$ with radial initial data. From \cite{HR,Zw}, there exist blow-up solutions with cylindrically symmetry blowing up at $\log$-$\log$ blow up rate for $p=3$ and $d\geq3$. In cite In \cite{MRS2}, Merle, Rapha\"{e}l and Szeftel construct a stable self-similar blow up dynamics for slightly supercritical nonlinearity in low dimensions (i.e. $d\leq 5$).
    
For supercritical gKdV equations the existence of blow up solution in energy space $H^1$ has been a long standing open problem. Numerical simulation \cite{DM} suggests the existence of self-similar blow up solution in the slightly $L^2$ supercritical case%
\footnote{From \cite{MM3}, there is no self-similar blow up solutions for $L^2$ critical gKdV with slightly supercritical mass.}
, where a self similar solution is a solution of the following form:
$$u(t,x)\sim\frac{1}{[3b(T-t)]^{\frac{2}{3(p-1)}}}V_b\bigg(\frac{x}{\sqrt[3]{3b(T-t)}}\bigg),$$
where $b>0$. Direct computation shows that $V_b$ should be a solution to the following ODE%
\footnote{See Section 1.6 for the definition of ``$\Lambda$".}%
:
\begin{equation}
\label{SSE}
b\Lambda V_b+(V_b''-V_b+V_b|V_b|^{p-1})'=0.
\end{equation}

The exact solution of \eqref{SSE} in slightly supercritical case has been constructed by H. Koch \cite{K}. It is related to an eigenvalue problem, i.e. for all $p>5$ close enough to $5$, there exists a unique $b=b(p)>0$ such that a unique solution $V_b$ to \eqref{SSE} can be found. Hence, this $V_b$ leads to a self-similar blow up solution to \eqref{CP} directly. 

But unfortunately, the exact solution $V_b$ constructed in \cite{K}, has a slowly decaying tail:
$$V_b(y)\sim\frac{1}{|y|^{\frac{1}{2}-\sigma_c}},\quad \text{as }|y|\rightarrow +\infty.$$

Thus, $V_b$ belongs to $L^{p+1}\cap \dot{H}^1$, but always misses the critical Sobolev space $\dot{H}^{\sigma_c}$ (hence $L^2$), which makes it impossible to obtain a stability result for the exact self-similar blow up solution. Since, for typical Cauchy problem like \eqref{CP}, we can only expect a stability result in a Cauchy space, i.e. a space where local wellposedness holds. In our case, natural Cauchy spaces are the critical Sobolev space $\dot{H}^{\sigma_c}$ and the energy space $H^1$ from \cite{KPV}, while $V_b$ is not in neither of them. Hence, we cannot use the profile $V_b$ directly. 

Despite the slowly decaying tail of $V_b$, we can choose a suitable cut-off of $V_b$ as an
approximation, such that it is bounded in $L^2$ with exponential decay on the right. Based on this approximate self-similar profile, Lan \cite{L1} has construct a stable self-similar blow-up dynamics for slightly $L^2$ supercritical gKdV:}
\begin{theorem}[Existence and stability of a self-similar blow-up dynamics]\label{PT}
There exists a $p^*>5$ such that for all $p\in(5,p^*)$, there exist constants $\delta(p)>0$ and $b^*(p)>0$ with
\begin{gather}
\lim_{p\rightarrow5}\delta(p)=0\\
{b^*(p)=\frac{4\pi^2}{\Gamma(1/4)^4}(p-5)+O(|p-5|^2), \text{ as }p\rightarrow5}
\end{gather}
and a nonempty open subset $\mathcal{O}_p$ in $H^1$ such that the following holds. If $u_0\in \mathcal{O}_p$, then the corresponding solution to \eqref{CP} blows up in finite time $0<T<+\infty$, with the following dynamics: there exist geometrical parameters $(\lambda(t),x(t))\in\mathbb{R}_{+}^*\times\mathbb{R} $ and an error term $\varepsilon(t)$ such that:
\begin{equation}\label{11}
u(t,x)=\frac{1}{\lambda(t)^{\frac{2}{p-1}}}\big{[}\mathcal{Q}_p+\varepsilon(t)\big{]}\bigg{(}\frac{x-x(t)}{\lambda(t)}\bigg{)}
\end{equation}
with
\begin{equation}\label{12}
\|\varepsilon_{y}(t)\|_{L^2}\leq \delta(p).
\end{equation}
Moreover, we have:
\begin{enumerate}
\item The blow-up point converges at the blow-up time:
\begin{equation}\label{13}
x(t)\rightarrow x(T) \text{ as } t\rightarrow T,
\end{equation}
\item The blow-up speed is self-similar:
\begin{equation}\label{14}
\forall t\in[0,T),\quad (1-\delta(p))\sqrt[3]{3b^*(p)}\leq\frac{\lambda(t)}{\sqrt[3]{T-t}}\leq(1+\delta(p))\sqrt[3]{3b^*(p)}.
\end{equation}
\item The following convergence holds:
\begin{equation}\label{15}
\forall q\in[2,q_c), \quad u(t)\rightarrow u^*\text{ in $L^q$ as $t\rightarrow T$},
\end{equation}
where $q_c=\frac{p-1}{2}$.
\item The asymptotic profile $u^*$ displays the following singular behavior:
\begin{equation}\label{16}
\big{(}1-\delta(p)\big{)}\int \mathcal{Q}^2_p\leq \frac{1}{R^{2\sigma_c}}\int_{|x-x(T)|<R}|u^*|^2\leq \big{(}1+\delta(p)\big{)}\int \mathcal{Q}^2_p.
\end{equation}
for $R$ small enough.
In particular, we have for all $ q\geq q_c$:
$$u^*\notin L^q.$$
\end{enumerate}
\end{theorem}

\subsection{Blow up solution with multiple blow up points} {The existence of large blow up solution is mostly open. One way to construct such solution is to consider blow up solutions with multiple blow up points. There are several examples:
\begin{itemize}
	\item Merle \cite{M5} for $L^2$ critical NLS with conformal blow up rate;
	\item Planchon, Rapha\"{e}l \cite{PR} and Fan \cite{Fan} for $L^2$ critical NLS with $\log$-$\log$ blow up rate;
	\item Merle \cite{M4} and Merle, Zaag \cite{MZ} for nonlinear heat equation with ODE blow up rate.
\end{itemize}
Such constructions correspond to the weak interaction case, i.e. the interaction between the bubbles does not change the blow up rate of each bubble. There are also some examples for strongly interacting bubbles:
\begin{itemize}
	\item Martel, Rapha\"{e}l \cite{MarR} for $L^2$ critical NLS;
	\item Cort\'{a}zar, Del Pino, Musso \cite{CDM} for energy critical nonlinear heat equations in domain;
	\item Jendrej \cite{J1,J2} for focusing energy critical wave equations.
\end{itemize} }

The goal of this paper is to construct blow up solutions for slightly supercritical gKdV with multiple bubbles, and each bubble concentrates at a finite point. First, we need to give the definition of the {\it blow up point} for solution to \eqref{CP}.

\begin{definition}
Let $u(t)$ be a solution of \eqref{CP}, which blows up in finite time $T$. The \emph{blow-up set} of $u(t)$ is the set of all the points $z$ such that:
$$\liminf_{R\rightarrow0}\bigg{(}\liminf_{t\rightarrow T}\frac{1}{R^{2\sigma_c}}\int_{|x-z|<R}|u(t)|^2\bigg{)}\geq \frac{1}{2}\int\mathcal{Q}^2_p.$$
\end{definition}
\begin{remark}\label{symmetry}
From the definition, the blow-up set is ``invariant" under the symmetry of the equation. More precisely, consider a solution $u(t)$ of \eqref{CP}, which blows-up in finite time $T$ with blow-up set $B$. Then for all $\lambda_0>0$, $x_0\in\mathbb{R}$,
$$\bar{u}(t,x)=\frac{1}{\lambda_0^{\frac{2}{p-1}}}u\bigg(\frac{t}{\lambda_0^3},\frac{x-x_0}{\lambda_0}\bigg),$$
is still a solution to \eqref{CP}, which blows up in finite time $\bar{T}=\lambda_0^3 T$. Moreover, its blow-up set is exactly:
$$\bar{B}=\big\{\lambda_0x+x_0|x\in B\big\}.$$
\end{remark}
\begin{remark}
For all solution $u(t)$ mentioned in Theorem \ref{PT}, we {can} see from the proof of \eqref{16} in \cite{L1},
\begin{gather*}
\limsup_{R\rightarrow0}\lim_{t\rightarrow T}\frac{1}{R^{2\sigma_c}}\int_{|x-z_0|<R}|u(t)|^2\leq \delta(p), \quad\text{for all $z_0\not=x(T)$},\\
\frac{1}{R^{2\sigma_c}}\int_{|x-x(T)|<R}|u(t)|^2\sim\int\mathcal{Q}_p^2, \quad\text{for $R$ small enough}.
\end{gather*}
Therefore, the blow up set of $u(t)$ is exactly $\{x(T)\}$.
\end{remark}
	
\subsection{Statement of the main theorem}
\begin{theorem}[Main Theorem]\label{MT}
There exist universal constants $p^*>5$, $c>0$ such that for all $p\in(5,p^*)$, $k\in\mathbb{N}^+$, if $2\leq k\leq c|\log(p-5)|$, then for all $k$ pairwise distinct points $x_1,x_2,\ldots,x_k\in \mathbb{R}$, {there exists $u_0\in C_0^{\infty}$, such that the corresponding solution $u(t)$ of \eqref{CP}, blows up in finite time $T<+\infty$}. And for $t$ close to $T$, there exist scaling parameters $\lambda_j(t)\in \mathbb{R}^+$ and an error term $\tilde{u}(t,x)$ with
\begin{equation}\label{10001}
u(t,x)=\sum_{j=1}^k\frac{1}{\lambda^{\frac{2}{p-1}}_j(t)}\mathcal{Q}_p\bigg(\frac{x-x_j}{\lambda_j(t)}\bigg)+\tilde{u}(t,x).
\end{equation}
Here for all $j=1,\ldots,k$,  and $t$ close to $T$,
\begin{gather}
(1-\delta(p))\sqrt[3]{b^*(p)}\leq\frac{\lambda_j(t)}{\sqrt[3]{T-t}} \leq (1+\delta(p))\sqrt[3]{b^*(p)}, \label{10002}\\
\lambda_j(t)^{1-\sigma_c}\|\tilde{u}_x(t)\|_{L^2}\leq \delta(p),\label{10003}
\end{gather}
where $\delta(p)$ and $b^*(p)$ are the universal constants introduced in Theorem \ref{PT}. Hence the blow-up rate is self-similar, i.e.
{\begin{equation}
\frac{c_0(p)k}{(T-t)^{(1-\sigma_c)/3}}\leq\|u_x(t)\|_{L^2}\leq \frac{C_0(p)k}{(T-t)^{(1-\sigma_c)/3}},
\end{equation}}
for $t$ close to $T$. Here $0<c_0(p)<C_0(p)$ are two constants depending only on $p$.

Moreover, the blow-up set of $u(t)$ is exactly $\{x_1,x_2,\ldots,x_k\}$.
\end{theorem}

\noindent{\it Comments on Theorem \ref{MT}:}

{{\it 1. Large blow up solutions.} For solutions constructed in Theorem \ref{MT}, we know from the proof of Theorem \ref{MT} that $\|u_0\|_{L^{q_c}}\sim k\|\mathcal{Q}_p\|_{L^{q_c}}$. For $p$ close enough to $5$, $c|\log(p-5)|$ is large, thus we prove the existence of blow up solutions with large initial data (i.e. the critical Lebesgue norm is comparable to $|\log(p-5)|$) for slightly supercritical gKdV equations.

{\it 2. Higher regularity for multiple bubble blow up solutions.} {The initial data of the solutions constructed in Theorem \ref{MT} is in $C_0^{\infty}$. However, the results of Theorem \ref{MT} hold true for $u_0\in \mathcal{O}_{k,p}\subset H^2$, where $\mathcal{O}_{k,p}$ is an infinite subset in $H^2$ containing functions which are not in $C_0^{\infty}$.}  Here, we will see in Section 3.1, in the multiple bubble case (i.e. $k\geq2$), the minimal requirement on the regularity of the initial data is $u_0\in H^2$. This is in contrast with the single bubble case, where $H^1$ is enough for the analysis%
\footnote{From Theorem \ref{PT}, in the single bubble case, the initial data set is a nonempty open subset in $H^1$, hence containing functions which are not in $H^2$.}%
. For the multiple bubble case, the $H^2$ regularity on the initial data is used to control the error term between the blow-up points%
\footnote{See the proof of \eqref{311}.}%
.

{\it 3. On the multiple bubble problem.} There are two approach to this kind of problem:\\
- The first approach is that we work from the the asymptotic expansion and move the time backward to obtain a suitable initial data. This approach is suitable when estimates are reversible (do not depend on whether the time move back or forward). Typical example for this approach is \cite{M5} for $L^2$ critical NLS.\\
- The second approach is that we work directly from the initial data to achieved the result. Since we have to deal with some instability directions (for example the translation of the blow-up point), we will need to adjust some finite dimensional parameters to have the right initial data (the so-called \emph{topological argument}). These approach is suitable when the estimates we are dealing with are only for time moving forward, due to a parabolic effect. We conclude the proof by adjusting the finite dimensional parameters using a Brouwer type theorem. There are several examples of this approach: \cite{Fan,MarR,M4,MZ} for multiple bubble blow up solutions, \cite{CMM,KMR,MMT1,MMT} for multi-soliton solutions.

While for supercritical gKdV, the estimates we are dealing with are only for time moving forward due to a hidden parabolic nature of the Airy equation, so it is natural to take the second approach. We mention here, due to the use of a topological argument, the solutions obtained in Theorem \ref{MT} is not expected to be stable.

{\it 4. Blow up speed.} The blow up solution constructed in Theorem \ref{MT} corresponds to the weak interaction case, i.e. the blow up speed is still self-similar, same as the single bubble case. The existence of blow up solution to supercritical gKdV with blow up rate other than self-similar still remains open.}

\subsection{Outline of the proof}
The main idea in this paper is to construct a solution which behaves like a decoupled sum of $k$ self-similar blow-up solutions constructed in Theorem \ref{PT}. To do this, we start with a nonempty open subset of initial data $\mathcal{U}_{k,p}\subset H^2$, consisting of $H^2$ functions which can be written as a decoupled sum of bubbles. Then we establish the geometrical decomposition and the modulation estimates for the corresponding solutions just like what we do in \cite{L1}. Next we use a topological argument to show that there exists a nonempty subset $\mathcal{O}_{k,p}\subset\mathcal{U}_{k,p}$, such that the corresponding solution has exactly $k$ blow-up points. { Here for technical reasons, we have to assume that the distance between the blow-up points is large.} Finally, by another topological argument and { a standard argument of scaling}, we can show that the blow-up points can be chosen arbitrarily. To be specific, we have the following steps:
\subsubsection{Geometrical decomposition and modulation estimate (Section 2)} { We start with initial data which can be written as a decoupled sum of bubbles plus a small error term, i.e.
\begin{equation}
\label{FD}
u_0(x)=\sum_{j=1}^k\frac{1}{\lambda_{j,0}^{\frac{2}{p-1}}}Q_{b_{j,0}}\bigg(\frac{x-x_{j,0}}{\lambda_{j,0}}\bigg)+\tilde{u}_0(x),
\end{equation}
where $Q_b$ is the self-similar profile constructed in \cite{L1}. Moreover, we assume that 
\begin{equation}\label{LDI}
\min_{1\leq i\not=j\leq k}|x_{i,0}-x_{j,0}|\geq b_c^{-100},
\end{equation}
where $b_c$ is some universal constant with $b_c\sim p-5>0.$}

We then apply the standard argument of implicit function theory to establish the following geometrical decomposition on some time interval $[0,T^*)$:
\begin{equation}\label{101}
u(t,x)=\sum_{i=1}^{k}\frac{1}{\lambda_i(t)^{\frac{2}{p-1}}}Q_{b_i(t)}\bigg{(}\frac{x-x_i(t)}{\lambda_i(t)}\bigg{)}+\tilde{u}(t,x),
\end{equation}
with some orthogonality conditions on the error term $\tilde{u}$. Here the formal ODE system of the parameters $(\lambda_j(t),b_j(t),x_j(t))$ is 
\begin{equation}
\label{FODE}
\begin{cases}
\frac{1}{\lambda_j}\frac{d\lambda_j}{ds_j}+b_j=0,\\
\frac{1}{\lambda_j}\frac{dx_j}{ds_j}-1=0,\\
\frac{db_j}{ds_j}=0,
\end{cases}
\quad\text{for all }j=1,\ldots,k.
\end{equation}
	
Following from similar arguments as in \cite{L1} and \cite{MMR1}, we can show the following modulation estimates (which can be viewed as an approximation of \eqref{FODE}) hold:
\begin{gather}
\bigg|\frac{1}{\lambda_j}\frac{d\lambda_j}{ds_j}+b_j\bigg|+\bigg|\frac{1}{\lambda_j}\frac{dx_j}{ds_j}-1\bigg|\lesssim b_c^2+\|\varepsilon_j\|_{H^1_{\rm loc}},\label{102}\\
\bigg|\frac{db_j}{ds_j}+c_p(b_j-b_c)b_c\bigg|\lesssim b_c^{\frac{5}{2}}+b_c\|\varepsilon_j\|_{H^1_{\rm loc}},\label{103}\\
s_j(t)=\int_0^t\frac{1}{\lambda^3_j(\tau)}\,d\tau,\quad \varepsilon_j(t,y)=\lambda_j^{\frac{2}{p-1}}(t)\tilde{u}\big(t,\lambda_j(t)y+x_j(t)\big),\nonumber
\end{gather}
for all $1\leq j\leq k$, as along as the geometrical decomposition \eqref{101} holds. 
	
\subsubsection{Construction of the initial data set (Section 2)} 
{In Section 2, we will construct the set of initial data which leads to multiple bubble blow-up solutions. We start from the formal ODE system \eqref{FODE}. We assume that $b_j(0)>0$ for all%
\footnote{This is ensured by the choice of the open set $\mathcal{U}_{k,p}$.}
$j=1,\ldots,k$. We can see the solution to \eqref{FODE} is:
$$\lambda_j(t)=\sqrt[3]{3b_j(0)(T_j-t)},\quad b_j(t)=b_j(0),$$
for all $j=1,\ldots,k$ and for all $t<\min_{j}T_j$, where
$$T_j=\frac{\lambda_j^3(0)}{3b_j(0)}>0.$$

Clearly, the solution has multiple blow-up points if and only if $T_1=\ldots=T_k$, or equivalently
\begin{equation}
\label{IC}
\frac{\lambda_1^3(0)}{3b_1(0)}=\cdots=\frac{\lambda_k^3(0)}{3b_k(0)}.
\end{equation}
That is to say a special condition on the initial data is needed to ensure the solution has exactly $k$ blow-up points. We can also see that the condition \eqref{IC} is not stable under small perturbation in $H^2$.

While for the approximation system \eqref{102} and \eqref{103}, it seems hard to find an explicit description (similar as \eqref{IC}) of the corresponding condition on the initial data. However, we may still use a standard topological argument to show the existence of such a condition. 

More precisely, we can find an infinite subset $\mathcal{O}_{k,p}\subset\mathcal{U}_{k,p}$ such that
\begin{equation}
\label{MBU}
\frac{1}{2^{k+1}}\leq \frac{\lambda_{i}(t)}{\lambda_j(t)}\leq 2^{k+1},
\end{equation}
for all $j=1,\ldots,k$ as long as the geometrical decomposition holds. {Here $\mathcal{O}_{k,p}$ is defined as following%
\footnote{See Section 2.4 for the definition of ``admissible".}
:
\begin{align*}
\mathcal{O}_{k,p}=\bigg\{&u_0=\frac{1}{\lambda_{1,0}^{\frac{2}{p-1}}}Q_{b_{1,0}}\bigg(\frac{x-x_{1,0}}{\lambda_{1,0}}\bigg)+\sum_{j=2}^k\frac{1}{(\lambda^*_{j,0})^{\frac{2}{p-1}}}Q_{b_{j,0}}\bigg(\frac{x-x_{j,0}}{\lambda^*_{j,0}}\bigg)+\tilde{u}_0(x),\\
&\qquad\Big|(\vec{b}_0,\vec{x}_0,\tilde{u}_0,\lambda_{1,0})\text{ is admissible.}\bigg\}
\end{align*}
where $\{\lambda_{j,0}^*\}_{j=2}^k$ depend continuously on $(\vec{b}_0,\vec{x}_0,\tilde{u}_0,\lambda_{1,0})$. Here $\tilde{u}_0$ is an $H^2$ function which is only assume to be small%
\footnote{Hence, one may chose $\tilde{u}_0\in C^{\infty}_0$, so that we have $u_0\in C^{\infty}_0$.}%
. }This is done by a standard topological argument%
\footnote{See Lemma \ref{Lemma21} for more details.}%
. 

From \eqref{MBU}, we can see that if the solution blows up in finite time, and the geometrical decomposition holds for all time, then the solution has exactly $k$ blow-up points. Hence, a good control on the error term is required.}

\subsubsection{Estimates on $\varepsilon_j$ by using monotonicity tools (Section 3).}{In Section 3, we will derive some crucial control of the error term $\varepsilon_j$, for all $1\leq j\leq k$. More precisely, for all%
\footnote{See Section 2.3 for the definition of $s_j^*$.}
 $s_j\in[0,s_j^*)$,
\begin{gather}
\int_{\kappa B<y<b_c^{-20}}\big(\varepsilon_j(s_j)\big)_y^2\lesssim b_c^{\frac{55}{7}},\label{104}\\
\frac{d}{ds_j}\mathcal{F}_j+\frac{\mu}{B}\|\varepsilon_j\|_{H^1_{\rm loc}}^2\lesssim b_c^{\frac{7}{2}},\label{105}
\end{gather}
where $\kappa,\mu>0$ are some universal constants, $B=b_c^{-1/20}$ is a large constant and 
$$\mathcal{F}_j= \int\bigg{[}(\varepsilon_j)_y^2\psi_B+\varepsilon_j^2\zeta_B-\frac{2}{p+1}\big{(}|\varepsilon_j+Q_{b_j}|^{p+1}-Q_{b_j}^{p+1}-(p+1)\varepsilon_j Q_{b_j}^p\big{)}\psi_B\bigg{]},$$
for some well chosen weight function $(\psi_B,\zeta_B)$.
	S
The derivation of these estimates follows from almost the same strategy and computation as in \cite[Section 4, Section 5]{L1}, which is developed originally in \cite{MM3} and \cite{MMR1}. The key observation is that the interaction of the bubbles:
$$\frac{1}{\lambda^{\frac{2}{p-1}}_j(t)}Q_{b_j(t)}\bigg(\frac{x-x_j(t)}{\lambda_j(t)}\bigg),$$
is extremely small due to the assumption of \eqref{LDI}. For all $j=1,\ldots,k$, we may ignore the bubbles with an index other than $j$, due to the choice of the weigh function. Then the estimate of the error term is exactly the same to the single blow-up point case. }
	
There are only two different things. One is that we need the $H^2$ assumption to estimate $\varepsilon_j$ on the interval between the blow-up points (clearly, there is no such interval in the single blow-up point case). The other one is that the error term $\tilde{u}$ behaves like a sum of $k$ error terms introduced in \cite{L1}. So if $k$ is too large, we cannot obtain the smallness of any global norm%
\footnote{For example, \eqref{BB5} and \eqref{BB6}.}
 of $\varepsilon_j$. That's why we need to add a restriction on $k$.
	
\subsubsection{End of the proof (Section 4 and Section 5).}{ Following from similar argument as in \cite[Section 6.1]{L1}, the modulation estimates \eqref{102}, \eqref{103} and the estimates on the error term obtained in Section 3, we can see that for $u_0\in\mathcal{O}_{k,p}$, the corresponding solution blows up in finite time $T$. We will also see that the translation parameters $\{x_1(t),\ldots,x_k(t)\}$ converge to $k$ pairwise distinct points $\{x_1(T),\ldots,x_k(T)\}$ as $t\rightarrow T$. Moreover, the blow-up set is exactly $\{x_1(T),\ldots,x_k(T)\}$.

Hence, we have already constructed solutions blow-up in finite time with exactly $k$ blow-up points, where the distance between the blow-up points is large%
\footnote{See \eqref{BB3} for more details.}%
. Then we can show that Theorem \ref{MT} follows from Proposition \ref{prop1} by standard arguments. 

Indeed for $k$ given pairwise distinct points $\{x_1,\ldots,x_k\}$, we first assume that the distance between them is large enough, i.e.
$$\min_{1\leq i\not=j\leq k}|x_i-x_j|\geq b_c^{-120}.$$
Based on the following two facts
\begin{enumerate}
\item the blow-up points are continuously depend on the initial data in $\mathcal{O}_{k,p}$ (due to the continuity of the functions $F_j$, $j=2,\ldots,k$);
\item the blow-up points are not too far away from the the translation parameters, i.e.
\begin{equation}
\max_{1\leq j\leq k}|x_j(0)-x_j(T)|\lesssim b_c^{-1},
\end{equation}
\end{enumerate}
we can construct blow-up solutions whose blow-up set is exactly the set of these $k$ points by a topological argument%
\footnote{See Lemma \ref{TL} for more details.}
(different from the one that is used to construct the set $\mathcal{O}_{k,p}$). 

While for arbitrarily given $k$ pairwise distinct points, from Remark \ref{symmetry}, we may use an argument of scaling to reduce to the case where the distance between the points is large. Thus, we conclude the proof of Theorem \ref{MT}.}

\subsection{Notations}
We first introduce the scaling generator:
\begin{equation}
\Lambda f=\frac{2}{p-1}f+yf'.
\end{equation}
	
We denote the $L^2$ scalar product by:
\begin{equation}
(f,g)=\int_{\mathbb{R}}f(x)g(x)dx
\end{equation}
and observe the integration by parts:
\begin{equation}
(\Lambda f,g)=-(f,\Lambda g+2\sigma_c g).
\end{equation}
	
Then we let $\mathcal{Q}_p$ be the ground state. For $p=5$, we simply write $\mathcal{Q}_p$ as $\mathcal{Q}$. We introduce the linearized operators at $\mathcal{Q}_p$:
\begin{equation}
Lf=-f''+f-p\mathcal{Q}_p^{p-1}f.
\end{equation}
A standard computation leads to:
\begin{equation}
L(\mathcal{Q}_p')=0,\quad L(\Lambda \mathcal{Q}_p)=-2\mathcal{Q}_p.
\end{equation}
{ We also denote by 
\begin{equation}
\nu=\frac{1}{1000}>0,
\end{equation}
a small universal constant.}

Next, we denote by $A\lesssim B$ ($A\gtrsim B$), if there exists a universal constant%
\footnote{In this paper, ``universal constant" means a constant which does not depend on $p$ and $k$.}
 $C>0$ such that
\begin{equation}
A\leq CB\text{ ($A\geq \frac{1}{C}B$)}.
\end{equation}
	
Finally, we denote by $\delta(p)$ a small positive constant such that:
\begin{equation}
\lim_{p\rightarrow5}\delta(p)=0.
\end{equation}

\subsection*{Acknowledgement} The author is grateful for his supervisors F. Merle \& T. Duyckaerts for suggesting this problem and giving a lot of guidance. The author also thanks H. Koch for very helpful discussion about the self-similar profile.
	
\section{Modulation estimate and Topological argument}
\subsection{Self-similar profile}
Let us first recall the properties of the self-similar profile $Q_b$ constructed in \cite{L1}:
\begin{proposition}[H. Koch, \cite{K}]
There exists $p^*>5$, $b^*>0$ and 2 smooth maps: $\gamma(b,p): [0,b^*)\times[5,p^*)\rightarrow \mathbb{R}$, $v(b,p,y): [0,b^*)\times[5,p^*)\times\mathbb{R}\rightarrow \mathbb{R}$, such that the following holds:
\begin{enumerate}
\item The self-similar equation:
\begin{gather}
b\big{(}(1+\gamma(b,p))v+xv'\big{)}+(v''-v+v|v|^{p-1})'=0,\label{SFE}\\
(v(b,p,\cdot),\mathcal{Q}_p'(\cdot))=0,\quad v(b,p,y)>0.
\end{gather}
\item For all $ p\in[5,p^*)$, there exists a unique $b=b(p)\in[0,b^*)$ such that:
\begin{equation}
\gamma(b(p),p)=-1+\frac{2}{p-1},\quad b(5)=0,
\end{equation}
Moreover,
\begin{gather}
\frac{d b(p)}{d p}\bigg{|}_{p=5}=\frac{\|\mathcal{Q}\|_{L^2}^2}{\|\mathcal{Q}\|_{L^1}^2}=\frac{4\pi^2}{\Gamma(1/4)^4}>0,\\
\frac{\partial\gamma}{\partial b}\bigg{|}_{b=b(p)}=-\frac{\|\mathcal{Q}_p\|_{L^1}^2}{8\|\mathcal{Q}_p\|_{L^2}^2}+O(|p-5|)<0,\\
\frac{1}{2}\int|v_y(b(p),p,y))|^2dy-\frac{1}{p+1}\int|v(b(p),p,y)|^{p+1}dy=0.
\end{gather}
\item $v(b,p,\cdot)\in \dot{H}^1\cap L^{p+1}$, $v(b,p,\cdot)\notin L^2$ if $b>0$ and $v(0,p,y)=\mathcal{Q}_p(y)$. Moreover, let
$$w_p(b,y)=v(b,p,y)-\mathcal{Q}_p(y),$$
then for all $k,n\in\mathbb{N}$ there holds:
\begin{gather}
|w_p(b,y)|\lesssim\begin{cases}
e^{-\frac{1}{3b}}(1+b^{-2/3}|1-by|)^{-1-\gamma} & \text{if $y > b^{-1}$},\\
b\exp(\frac{1}{3b}[(1-by)^{3/2}-1]) & \text{if $b^{-1}\geq y>0$},\\
b(1-by)^{-1-\gamma} & \text{if $y\leq0$},
\end{cases}\label{ASB}\\
|\partial_y^k\partial_b^nv|
\lesssim\begin{cases}
e^{-\frac{1}{3b}}(1+b^{-2/3}|1-by|)^{-1-\gamma-k} & \text{if $y > b^{-1}$},\\
e^{-\frac{y}{10}}& \text{if $0\leq y\leq b^{-1}$},\\
\big{|}\partial_y^k\partial_b^n\big{(}b(1-by)^{-1-\gamma}\big{)}\big{|}+e^y & \text{if $y\leq0$}.
\end{cases}\label{ASB1}
\end{gather}
\end{enumerate}	
\end{proposition}
For $p^*>p>5$ fixed, we denote by $b_c=b(p)\sim p-5$. We choose a smooth cut-off function $\chi$, such that $\chi(y)=0$ if $|y|>2$, $\chi(y)=1$ if $|y|<1$. We define the localized profile as follows:
$$Q_b(y)=v(b,p,y)\chi(b_cy).$$
	
Then $Q_b$ has the following properties:
\begin{lemma}[Properties of the localized profile]\label{SSP} Assume that $b_c$ is small and $|\tilde{b}|\ll b_c$, then there holds:
\begin{enumerate}
\item Estimates on $Q_b$, for all $ \ell\in\mathbb{N}$, $q\in [1,+\infty]$:
\begin{align}
&|\partial^{\ell}_yQ_b(y)|\lesssim_{\ell} e^{-\frac{y}{10}}, \quad\text{for } y\geq 0,\\
&|\partial_y^{\ell} Q_b(y)|\lesssim_{\ell} e^{y}+b_c^{1+k}\mathbf{1}_{[-2b_c^{-1},0]}(y),\quad\text{for } y\leq 0,\\
&\|Q_b-\mathcal{Q}_p\|_{L^q}\lesssim b_c^{1-\frac{1}{q}},\quad \|(Q_b-\mathcal{Q}_p)_y\|_{L^2}\lesssim b_c.
\end{align}
Here $\mathbf{1}_I$ is the characteristic function of any interval $I$.
\item $Q_b$ is an approximate self-similar profile: Let
\begin{equation}
-\Phi_b=b\Lambda Q_b+(Q_b''-Q_b+Q_b|Q_b|^{p-1})',
\end{equation}
then for $\ell=0,1$:
\begin{equation}\label{APP}
\partial_y^{\ell}\Phi_b=C_p\tilde{b}b_c\partial_y^{\ell}Q_b+O\big{(}|\tilde{b}|^2\partial_y^{\ell}Q_b+b_c^{2}\mathbf{1}_{[-2,-1]}(b_cy)+e^{-\frac{1}{10b_c}}\mathbf{1}_{[1,2]}(b_cy)\big{)},
\end{equation}
where $C_p=\frac{d\gamma}{db}\big{|}_{b=b_c}<0$.
\item Energy property of $Q_b$:
\begin{equation}
|E(Q_b)|\lesssim b_c^3+|\tilde{b}|.
\end{equation}
\item Properties of the first order term with respect to $b$:
let $P_b(y)=\frac{\partial Q_b}{\partial b}(y)$, then
\begin{equation}
|P_b(y)|\lesssim e^{-\frac{y}{10}}\mathbf{1}_{\{y>0\}}(y)+\mathbf{1}_{[-2b_c^{-1},0]}(y).
\end{equation}
Furthermore, we have:
\begin{equation}\label{NV}
(P_b,\mathcal{Q}_p)=\frac{1}{16}\bigg{(}\int \mathcal{Q}_p\bigg{)}^2+O(|p-5|)>0.
\end{equation}
\end{enumerate}
\end{lemma}
	
\subsection{Geometrical decomposition}
We first give definition of the open subset $\mathcal{U}_{k,p}$ such that the corresponding solution has at least one blow-up point. 
\begin{definition}
Let $p^*>5$ and close enough to $5$, for all $p\in(5,p^*)$ we define $\mathcal{U}_{k,p}$ as the set of all $u_0$ satisfying the following conditions:
\begin{enumerate}
\item Geometrical decomposition:
\begin{equation}
u_0(x)=\sum_{i=1}^k\frac{1}{\lambda_{i,0}^{\frac{2}{p-1}}}Q_{b_{i,0}}\bigg(\frac{x-x_{i,0}}{\lambda_{i,0}}\bigg){}+\tilde{u}_0(x).
\end{equation}
\item $b_{j,0}$ is near $b_c$:
\begin{equation}\label{27}
|b_{j,0}-b_c|\leq b_c^5.
\end{equation}
\item Conditions on the scaling parameters:
\begin{gather}
0<\lambda_{j,0}<1,\\
\frac{1}{2^k}<\frac{\lambda_{i,0}}{\lambda_{j,0}}<2^k,\quad \text{for all }i\not=j\label{28}.
\end{gather}
\item Distance between the blow-up points:
\begin{equation}
|x_{i,0}-x_{j,0}|> b_c^{-100},\quad \text{for all }i\not=j.
\end{equation}
\item $H^2$ smallness on the error term:
\begin{equation}\label{29}
\|\tilde{u}_0\|_{H^2}< b_c^{50}
\end{equation}
\end{enumerate}
\end{definition}
	
\begin{remark}
It is easy to check that $\mathcal{U}_{k,p}$ is nonempty and open.
\end{remark}
	
\begin{remark}
We will see in the next section why the assumption $\mathcal{U}_{k,p}\subset H^2$ is necessary for $k\geq 2$.
\end{remark}
Now we can introduce the classical geometrical decomposition. From a standard inverse function argument%
\footnote{See Lemma 1 in \cite{MM4} and Lemma 2.5 in \cite{MMR1}.}%
, we know that for all $u_0\in \mathcal{U}_{k,p}$, there exist a $t^*>0$ and geometrical parameters $\lambda_i(t)$, $b_i(t)$, $x_i(t)$, such that the corresponding solution $u(t)$ satisfies the following for all $t\in[0,t^*)$:
\begin{enumerate}
\item Geometrical decomposition:
\begin{equation}\label{GD}
u(t,x)=\sum_{i=1}^{k}\frac{1}{\lambda_i(t)^{\frac{2}{p-1}}}Q_{b_i(t)}\bigg(\frac{x-x_i(t)}{\lambda_i(t)}\bigg)+\tilde{u}(t,x).
\end{equation}
\item Orthogonality condition:
\begin{equation}\label{OC}
(\varepsilon_j(t),\mathcal{Q}_p)=(\varepsilon_j(t),\Lambda\mathcal{Q}_p)=(\varepsilon_j(t),y\Lambda\mathcal{Q}_p)=0,\text{ for all }j=1,\cdots,k,
\end{equation}
where $\mathcal{Q}_p$ is the ground state and 
$$\varepsilon_j(t,y)=\lambda_j(t)^{\frac{2}{p-1}}\tilde{u}\big(t,\lambda_j(t)y+x_j(t)\big).$$
\item Estimates on the parameters at the initial time: for all $i\not=j$,
\begin{gather}
0<\lambda_j(0)<2,\label{21}\\
\frac{2\lambda_{i,0}}{3\lambda_{j,0}}<\frac{\lambda_i(0)}{\lambda_j(0)}<\frac{3\lambda_{i,0}}{2\lambda_{j,0}},\label{22}\\
|b_j(0)-b_c|\leq b_c^{4},\label{225}\\
|x_i(0)-x_j(0)|\geq b_c^{-80},\label{23}\\
\|\varepsilon_j(0)\|_{L^{(p-1)/2}}+\big\|(\varepsilon_j(0))_y\big\|_{H^1}\leq b_c^{30}\label{24}
\end{gather}
\item Continuity of the parameters: Consider $u_{0,n}\in\mathcal{U}_{k,p}$, $u_{0,n}\rightarrow u_0$ in $H^1$. Let $u_n(t)$ be the solution of \eqref{CP} with initial data $u_{0,n}$ and $\lambda_{j,n}(t)$, $b_{j,n}(t)$, $x_{j,n}(t)$, $\tilde{u}_{n}(t,x)$ be the corresponding geometrical parameters and error terms. Suppose there exists a $t^*_1>0$ such that the geometrical decomposition for all $u_n(t)$ and $u(t)$ hold on $[0,t^*_1)$, then for all $0\leq t<t^*_1$, we have:
\begin{equation}\label{continuity}
\lim_{n\rightarrow +\infty}\big(\lambda_{j,n}(t),b_{j,n}(t),x_{j,n}(t),\tilde{u}_{n}(t)\big)=\big(\lambda_j(t),b_j(t),x_j(t),\tilde{u}(t)\big).
\end{equation}
\end{enumerate}
	
Next, we want to define the localized $H^1$ norm of $\varepsilon_j$. We first denote:
\begin{equation}
B=b_c^{-\frac{1}{20}},
\end{equation}
and choose a smooth weight function $\varphi$ such that:
\begin{equation}
\begin{split}
&\varphi(y)=
\begin{cases}
e^y &\text{ if } y<-1,\\
+y &\text{ if } -\kappa<y<\kappa,\\
3   &\text{ if } y>1.\\
\end{cases}\\
&\varphi'\geq 0\text{ for all }y\in\mathbb{R},
\end{split}
\end{equation}
where $\kappa$ is a small universal positive constant which will be chosen later.
	
Then we define the following {localized $H^1$ norm}:
\begin{equation}
\mathcal{N}_{j}(t)=\int\big((\varepsilon_j)^2_y+\varepsilon_j^2\big)\varphi'_B,
\end{equation}
where $\varphi_B(y)=\varphi(y/B).$
	
Let us consider the maximal time $T^*$ such that the geometrical decomposition \eqref{GD}, orthogonality condition \eqref{OC} and the following a $priori$ estimates hold in $[0,T^*)$:
\begin{gather}
0<\lambda_j(t)<3,\label{BA1}\\
\frac{1}{2^{k+2}}<\frac{\lambda_i(t)}{\lambda_j(t)}<2^{k+2},\label{BA2}\\
|b_{j}(t)-b_c|\leq b_c^{\frac{3}{2}+\nu},\label{BA3}\\
|x_i(t)-x_j(t)|\geq b_c^{-70}, \text{ for all }i\not=j,\label{BA4}\\
\|\varepsilon_j(t)\|_{L^{p_0}}\leq b_c^{\frac{23}{50}},\label{BA5}\\
\big\|\big(\varepsilon_j(t)\big)_y\big\|_{L^2}\leq b_c^{\frac{2}{3}},\label{BA6}\\
\mathcal{N}_j(t)\leq b_c^{3+6\nu},\label{BA7}
\end{gather}
where $\nu>0$ is a small universal constant to be chosen later, and $$p_0=\frac{5}{2}.$$

\begin{remark}
It is easy to see from \eqref{21}--\eqref{24} and the continuity of the flow that $T^*>0$.
\end{remark}	
	
{\begin{remark}\label{rem1}
Our goal is to improve these estimates in $[0,T^*)$. Then from a standard bootstrap argument, we have $T^*=T$,  and these estimates actually hold on $[0,T)$, where $T$ is the maximal life span. Indeed, following from similar argument as in \cite{L1}, we can improve \eqref{BA1} and \eqref{BA3}--\eqref{BA7}. But to improve the bound \eqref{BA2}, we need to assume that $u_0\in\mathcal{O}_{k,p}\subset\mathcal{U}_{k,p}$, where $\mathcal{O}_{k,p}$ is an infinite subset of $\mathcal{U}_{k,p}$. This subset can be constructed by a topological argument.
\end{remark}}

\begin{remark}
From \eqref{BA5}, \eqref{BA6} and Gagliardo-Nirenberg's inequality, we have for all $q_0\geq p_0$, $t\in[0,T^*)$,
\begin{equation}\label{210}
\|\varepsilon_j(t)\|_{L^{q_0}}\leq b_c^{\frac{149q_0-62}{270q_0}}.
\end{equation}
\end{remark}

\subsection{Modulation estimate}
In this subsection, we will prove the modulation estimates for the geometrical parameters on $[0,T^*)$ by using the a $priori$ estimates \eqref{BA1}--\eqref{BA7}. We first introduce a rescaled coordinate $(s_j,y)$ for all $j=1,\ldots,k$:
$$s_j=\int_0^t\frac{1}{\lambda_j(\tau)^3}\,d\tau,\quad y=\frac{x-x_j(t)}{\lambda_j(t)}.$$
Let 
$$s_j^*=\int_0^{T^*}\frac{1}{\lambda_j(\tau)^3}\,d\tau.$$
Now we can state the modulation estimates:
\begin{proposition}
	For all $j=1,\ldots,k$, the following properties hold for all $s_j\in[0,s_j^*)$:
	\begin{enumerate}
		\item Equation of $\varepsilon_j$:
		\begin{equation}\label{ME}\begin{split}
		&\frac{d\varepsilon_j}{ds_j}=(L\varepsilon_j)_y-b_j\Lambda\varepsilon_j+\bigg(\frac{1}{\lambda_j}\frac{d\lambda_j}{ds_j}+b_j\bigg)(\Lambda Q_{b_j}+\Lambda \varepsilon_j)+\bigg(\frac{1}{\lambda_j}\frac{dx_j}{ds_j}-1\bigg)(Q_{b_j}+\varepsilon_j)_y\\
		&+\Phi_{b_j}-\frac{db_j}{ds_j}P_{b_j}-(R_{b_j}(\varepsilon_j))_y-(R_{NL}(\varepsilon_j))_y\\
		&+\sum_{i=1,i\not=j}^k\bigg(\frac{\lambda_j}{\lambda_i}\bigg)^{\frac{3p-1}{p-1}}\Bigg[\Phi_{b_i}\bigg(\frac{\lambda_j\cdot+x_j-x_i}{\lambda_i}\bigg)+\frac{db_i}{ds_i}P_{b_i}\bigg(\frac{\lambda_j\cdot+x_j-x_i}{\lambda_i}\bigg)\Bigg]\\
		&+\sum_{i=1,i\not=j}^k\bigg(\frac{\lambda_j}{\lambda_i}\bigg)^{\frac{3p-1}{p-1}}\Bigg[\bigg(\frac{1}{\lambda_i}\frac{d\lambda_i}{ds_i}+b_i\bigg)\Lambda Q_{b_i}+\bigg(\frac{1}{\lambda_i}\frac{dx_i}{ds_i}-1\bigg)(Q_{b_i})_y\Bigg]\bigg(\frac{\lambda_j\cdot+x_j-x_i}{\lambda_i}\bigg)\\
		&-p\sum_{i=1,i\not=j}^k\bigg(\frac{\lambda_j}{\lambda_i}\bigg)^{2+\frac{2}{p-1}}\Bigg[\varepsilon_jQ_{b_i}^{p-1}\bigg(\frac{\lambda_j\cdot+x_j-x_i}{\lambda_i}\bigg)\Bigg]_y,
		\end{split}
		\end{equation}
		where
		\begin{align*}
		&\Phi_{b_i}=-b_i\Lambda Q_{b_i}-(Q_{b_i}''-Q_{b_i}+Q_{b_i}^p)',\\
		&R_{b_j}(\varepsilon_j)=p(Q_{b_j}^{p-1}-\mathcal{Q}_p^{p-1})\varepsilon_j,\\
		&R_{NL}(\varepsilon_j)=\bigg[\varepsilon_j+\sum_{i=1}^k\bigg(\frac{\lambda_j}{\lambda_i}\bigg)^{\frac{2}{p-1}}Q_{b_i}\bigg(\frac{\lambda_j\cdot+x_j-x_i}{\lambda_i}\bigg)\bigg]^p\\
		&\quad\qquad\qquad -p\sum_{i=1}^k\bigg(\frac{\lambda_j}{\lambda_i}\bigg)^{2+\frac{2}{p-1}}\varepsilon_jQ_{b_i}^{p-1}\bigg(\frac{\lambda_j\cdot+x_j-x_i}{\lambda_i}\bigg)\\
		&\quad\qquad\qquad-\sum_{i=1}^k\bigg(\frac{\lambda_j}{\lambda_i}\bigg)^{2+\frac{2}{p-1}}Q_{b_i}^p\bigg(\frac{\lambda_j\cdot+x_j-x_i}{\lambda_i}\bigg).
		\end{align*}
		\item Modulation estimates:
		\begin{gather}
		\bigg|\frac{1}{\lambda_j}\frac{d\lambda_j}{ds_j}+b_j\bigg|\lesssim b_c^2+\mathcal{N}_j^{\frac{1}{2}},\label{ME1}\\
		\bigg|\frac{1}{\lambda_j}\frac{dx_j}{ds_j}-1\bigg|\lesssim b_c^2+\mathcal{N}_j^{\frac{1}{2}},\label{ME2}\\
		\bigg|\frac{db_j}{ds_j}+c_p\tilde{b}_jb_c\bigg|\lesssim b_c^3+b_c\mathcal{N}_j^{\frac{1}{2}},\label{ME3}
		\end{gather}
		where $c_p=2+O(|p-5|)>0$.
	\end{enumerate}
\end{proposition}
\begin{proof}
	The proof of this proposition is almost the same as Proposition 3.1 in \cite{L1}. The only difference here is that we need to deal with some terms like 
	\begin{equation}
	Q_{b_i}\bigg(\frac{\lambda_j\cdot+x_j-x_i}{\lambda_i}\bigg),\,P_{b_i}\bigg(\frac{\lambda_j\cdot+x_j-x_i}{\lambda_i}\bigg) \text{ or } \Phi_{b_i}\bigg(\frac{\lambda_j\cdot+x_j-x_i}{\lambda_i}\bigg), 
	\end{equation}
	for $i\not=j$.
	
	We consider for example the following term 
	\begin{equation}\label{25}
	-p\sum_{i=1,i\not=j}^k\bigg(\frac{\lambda_j}{\lambda_i}\bigg)^{2+\frac{2}{p-1}}\Bigg[\varepsilon_jQ_{b_i}^{p-1}\bigg(\frac{\lambda_j\cdot+x_j-x_i}{\lambda_i}\bigg)\Bigg]_y.
	\end{equation}
	Since $Q_{b_i}$ is supported in $[-2b_c^{-1},2b_c^{-1}]$, if $y$ belongs to the support of
	$$Q_{b_i}^{p-1}\bigg(\frac{\lambda_j\cdot+x_j-x_i}{\lambda_i}\bigg),$$
	then we have 
	$$\frac{x_i-x_j}{\lambda_j}-2b_c^{-1}\frac{\lambda_i}{\lambda_j}<y<\frac{x_i-x_j}{\lambda_j}+2b_c^{-1}\frac{\lambda_i}{\lambda_j}.$$
	From \eqref{BA2} and \eqref{BA4}, we know that if $i\not=j$, 
	$$\bigg|\frac{x_i-x_j}{\lambda_j}\pm2b_c^{-1}\frac{\lambda_i}{\lambda_j}\bigg|\geq b_c^{-70}-2b_c^{-1}2^{k+2}>b_c^{-10},$$
	provided that $10^k\leq b_c^{-\nu/4}$ for some small universal constant $\nu>0$.  Since $$b_c\sim p-5,$$ this is implied by the condition 
	$$k\leq c|\log(p-5)|,$$
	if we choose 
	$$c=\frac{\nu}{8\log 10}=\frac{1}{8000\log 10}>0.$$
	
	Since we are considering a scalar product of \eqref{25} and some functions with exponential decay (i.e. $\mathcal{Q}_p$, $\Lambda\mathcal{Q}_p$, $y\Lambda{\mathcal{Q}_p}$), these terms can be controlled by
	$$\sum_{i=1,i\not=j}^k 10^k \exp(-\frac{1}{10}b_c^{-10})\leq b_c^{10},$$
	provided that $10^k\leq b_c^{-\nu/4}$. Then we conclude the proof.
\end{proof}

\subsection{First topological argument}
In this subsection we will find a nonempty subset $\mathcal{O}_{k,p}\subset\mathcal{U}_{k,p}$, such that the corresponding scaling parameters (i.e. $\lambda_j(t)$) are comparable to each other.

\begin{proposition}\label{211}
There exists a nonempty subset $\mathcal{O}_{k,p}\subset\mathcal{U}_{k,p}$, which {contains infinite many elements}, such that for all solution $u(t)$ with initial data in $\mathcal{O}_{k,p}$, the corresponding scaling parameters $\lambda_j(t)$ satisfy:
\begin{equation}
\frac{1}{2^{k+1}}\leq\frac{\lambda_i(t)}{\lambda_j(t)}\leq2^{k+1},\text{ for all $t\in [0,T^*)$ and $1\leq i,j\leq k$}.
\end{equation}
\end{proposition}

\begin{proof}
	We first claim the following lemma:
	\begin{lemma}\label{Lemma21}
		For all $t_0\in[0,T^*)$, $1\leq i\not=j\leq k$, if $\lambda_i(t_0)/\lambda_j(t_0)\geq \frac{10}{9}$, then for all $t\in[t_0,T^*)$, we have:
		{\begin{equation}\label{26}
		\frac{\lambda_i(t)}{\lambda_j(t)}\geq \frac{10}{9}.
		\end{equation}}
	\end{lemma}
	\begin{proof}[Proof of Lemma \ref{Lemma21}]The proof of \eqref{26} is a consequence of the a $priori$ assumption \eqref{BA1}--\eqref{BA7} and the modulation estimates. Indeed, from \eqref{BA3}, \eqref{BA7} and \eqref{ME1}, we have:
		$$-\frac{99}{100}b_c>(\lambda_j)_t\lambda_j^2>-\frac{101}{100}b_c.$$
	Then we can compute the derivative of $\lambda_i/\lambda_j$ with respect to $t$:
	\begin{align*}
	\frac{d}{dt}\bigg(\frac{\lambda_i}{\lambda_j}\bigg)&=\frac{(\lambda_i)_t\lambda_j-(\lambda_j)_t\lambda_i}{\lambda_j^2}\\
	&<\frac{1}{\lambda_j^2}\bigg(-\frac{99b_c\lambda_j}{100\lambda_i^2}+\frac{101b_c\lambda_i}{100\lambda_j^2}\bigg)\\
	&=\frac{b_c}{\lambda_j\lambda_i^2}\bigg[-\frac{99}{100}+\frac{101}{100}\bigg(\frac{\lambda_i}{\lambda_j}\bigg)^3\bigg].
	\end{align*}
	Similarly, we have:
	$$\frac{d}{dt}\bigg(\frac{\lambda_i}{\lambda_j}\bigg)>\frac{b_c}{\lambda_j\lambda_i^2}\bigg[-\frac{101}{100}+\frac{99}{100}\bigg(\frac{\lambda_i}{\lambda_j}\bigg)^3\bigg].$$
	The above {two} inequalities show that if for some time $t_0\in[0,T^*)$, we have
	$${\frac{\lambda_i(t_0)}{\lambda_j(t_0)}\geq \frac{10}{9},}$$
	then we have:
	$${\frac{d}{dt}\bigg(\frac{\lambda_i}{\lambda_j}\bigg)\bigg|_{t=t_0}>0.}$$
	Hence, the lemma follows from a simple argument.
	\end{proof}
	
For convenience we introduce the following notations:
\begin{enumerate}
	\item 
    $\vec{\lambda}_0=(\lambda_{1,0},\ldots,\lambda_{k,0})$, $\vec{b}_0=(b_{1,0},\ldots,b_{k,0})$, $\vec{x}_0=(x_{1,0},\ldots,x_{k,0})$, and
    \begin{gather*}
    F(\vec{\lambda}_0,\vec{b}_0,\vec{x}_0,\tilde{u}_0)=\sum_{i=1}^k\frac{1}{\lambda_{i,0}^{\frac{2}{p-1}}}Q_{b_{i,0}}\bigg(\frac{x-x_{i,0}}{\lambda_{i,0}}\bigg){}+\tilde{u}_0(x),
    \end{gather*}
    \item Let $\mathcal{C}$ be the set of $(\vec{\lambda}_0,\vec{b}_0,\vec{x}_0,\tilde{u}_0)$ such that \eqref{27}--\eqref{29} hold (or equivalently $F(\vec{\lambda}_0,\vec{b}_0,\vec{x}_0,\tilde{u}_0)\in \mathcal{U}_{k,p}$). 
    \item For $\ell=1,\ldots,k$, we let $\mathcal{C}_{\ell}$ be the set of $(\vec{\lambda}_0,\vec{b}_0,\vec{x}_0,\tilde{u}_0)$ such that \eqref{27}--\eqref{29} hold with \eqref{28} replaced by
    $$\frac{1}{2^{k+1-\ell}}<\frac{\lambda_{i,0}}{\lambda_{j,0}}<2^{k+1-\ell}.$$
\end{enumerate}
Clearly, we have $\mathcal{C}_k\subset\mathcal{C}_{k-1}\subset,\ldots,\subset\mathcal{C}_1=\mathcal{C}$.
Proposition \ref{211} is a simple consequence of the following lemma
\begin{lemma}\label{212}
	{For all $2\leq \ell\leq k$ there exist continuous functions $F_{\ell}$: 
	\begin{gather*}\mathbb{R}_{+}^k\times\mathbb{R}^k\times H^2\times \mathbb{R}_{+}^{k-\ell+1}\rightarrow\mathbb{R}_+,\\
	(\vec{b}_0,\vec{x}_0,\tilde{u}_0,\lambda_{1,0},\lambda_{\ell+1,0},\ldots,\lambda_{k,0})\mapsto F_{\ell}(\vec{b}_0,\vec{x}_0,\tilde{u}_0,\lambda_{1,0},\lambda_{\ell+1,0},\ldots,\lambda_{k,0}),
	\end{gather*}}such that for all $2\leq j\leq k$ and $(\lambda_{1,0},\lambda_{j+1,0},\ldots,\lambda_{k,0},\vec{b}_0,\vec{x}_0,\tilde{u}_0)$, if there exist $\lambda_{2,0},\ldots,\lambda_{j,0}>0$ such that
	$$(\lambda_{1,0},\lambda_{2,0},\ldots,\lambda_{k,0},\vec{b}_0,\vec{x}_0,\tilde{u}_0)\in\mathcal{C}_{j},$$
	then the following holds:
	\begin{enumerate}
		\item $(\lambda_{1,0},\lambda_{2,0}^*,\ldots,\lambda_{j,0}^*,\lambda_{j+1,0},\ldots,\lambda_{k,0},\vec{b}_0,\vec{x}_0,\tilde{u}_0)\in\mathcal{C}_{1}$, where
		{\begin{align*}
		&\lambda_{j,0}^*=F_j(\vec{b}_0,\vec{x}_0,\tilde{u}_0,\lambda_{1,0},\lambda_{j+1,0},\ldots,\lambda_{k,0})>0,\\
		&\lambda_{j-1,0}^*=F_{j-1}(\vec{b}_0,\vec{x}_0,\tilde{u}_0,\lambda_{1,0},\lambda_{j,0}^*,\lambda_{j+1,0},\ldots,\lambda_{k,0})>0,\\
		&\quad\vdots\\
		&\lambda_{2,0}^*=F_{2}(\vec{b}_0,\vec{x}_0,\tilde{u}_0,\lambda_{1,0},\lambda_{3,0}^*,\ldots\lambda_{j,0}^*,\lambda_{j+1,0},\ldots,\lambda_{k,0})>0.
		\end{align*}}
		\item Let $u(t)$ be the solution of \eqref{CP} with initial data $$u_0=F(\lambda_{1,0},\lambda_{2,0}^*,\ldots,\lambda_{j,0}^*,\lambda_{j+1,0},\ldots,\lambda_{k,0},\vec{b}_0,\vec{x}_0,\tilde{u}_0),$$ and $\{\lambda_{i}(t)\}_{i=1}^k$ be the corresponding scaling parameters, then for all $t\in[0,T^*)$, $1\leq i_1,i_2\leq j$, we have
		$$\frac{1}{2^{k+1}}\leq \frac{\lambda_{i_1}(t)}{\lambda_{i_2}(t)}\leq 2^{k+1}.$$
	\end{enumerate}
	
\end{lemma}
\begin{proof}
We will prove Lemma \ref{212} by induction on $j$. We first prove Lemma \ref{212} for $j=2$.

Consider $(\lambda_{1,0},\lambda_{3,0},\ldots,\lambda_{k,0},\vec{b}_0,\vec{x}_0,\tilde{u}_0)$ such that there exists $\lambda_{2,0}>0$ such that
$$(\lambda_{1,0},\lambda_{2,0},\ldots,\lambda_{k,0},\vec{b}_0,\vec{x}_0,\tilde{u}_0)\in\mathcal{C}_2.$$
We denote by $S_2$, the set of all $\lambda_{2,0}$ such that $(\vec{\lambda}_0,\vec{b}_0,\vec{x}_0,\tilde{u}_0)\in\mathcal{C}_1$. Clearly $S_2$ is a nonempty interval. Next we define the following sets:
\begin{align*}
S^<_2=&\{\lambda_{2,0}\in S_2|\text{The solution $u(t)$ with initial data $F(\vec{\lambda}_0,\vec{b}_0,\vec{x}_0,\tilde{u}_0)$,} \\
&\text{satisfies $\lambda_{2}(t_0)/\lambda_1(t_0)<\frac{1}{2^{k+1}}$, for some $t_0\in[0,T^*)$.}\},\\
S^>_2=&\{\lambda_{2,0}\in S_2|\text{The solution $u(t)$ with initial data $F(\vec{\lambda}_0,\vec{b}_0,\vec{x}_0,\tilde{u}_0)$,} \\
&\text{satisfies $\lambda_{2}(t_0)/\lambda_1(t_0)>2^{k+1}$, for some $t_0\in[0,T^*)$.}\}.
\end{align*}
For these two sets, we have the following observations:
\begin{enumerate}
	\item $S^<_2$ and $S^>_2$ are both contained in $S_2$ and open. Here the openness comes from \eqref{continuity}.
	\item $S^<_2\cap S^>_2$ is empty. This is a direct corollary of Lemma \ref{Lemma21}.
	\item $S^<_2\not=S_2$, $S^>_2\not=S_2$. Since if $S^<_2=S_2$, then from the definition of $\mathcal{C}_1$ and $\mathcal{C}_2$, we may always find some $\lambda_{2,0}\in S_2$ such that 
	$$\frac{\lambda_{2,0}}{\lambda_{1,0}}>2.$$
	Then from \eqref{22}, we have 	$$\frac{\lambda_2(0)}{\lambda_{1}(0)}>\frac{2}{3}\times\frac{\lambda_{2,0}}{\lambda_{1,0}} >\frac{10}{9}.$$ 
	then from Lemma \ref{Lemma21}, we have for all $t\in[0,T^*)$,
	$$\frac{\lambda_2(t)}{\lambda_1(t)}>\frac{10}{9},$$
	which leads to a contradiction. Similarly, we have $S^>_2\not=S_2.$
\end{enumerate}

Since $S_2$ is a nonempty interval (i.e. connected), the above observations imply that $S_2/(S^<_2\cup S^>_2)$ is not empty. On the other hand, it is easy to check that if $\delta>0$ is small enough, then $(\inf S_2+\delta)\in S^<_2$. So we have $\inf{S_2/(S^<_2\cup S^>_2)}\in S_2/(S^<_2\cup S^>_2)$. We then choose 
$$F_2(\vec{b}_0,\vec{x}_0,\tilde{u}_0,\lambda_{1,0},\lambda_{3,0},\ldots,\lambda_{k,0})=\inf{S_2/(S^<_2\cup S^>_2)}.$$
From Lemma \ref{Lemma21} and \eqref{continuity}, we know that $F_2$ is continuous.

Next (if $k\geq3$), suppose for all $2\leq j\leq j_0-1$ ($3\leq j_0\leq k$), Lemma \ref{212} holds, i.e. there exists a continuous function $F_2,\ldots,F_j$, such that for all $$(\lambda_{1,0},\lambda_{j+1,0},\ldots,\lambda_{k,0},\vec{b}_0,\vec{x}_0,\tilde{u}_0),$$ if there exist $\lambda_{2,0},\ldots,\lambda_{j,0}>0$ such that
$$(\lambda_{1,0},\lambda_{2,0},\ldots,\lambda_{k,0},\vec{b}_0,\vec{x}_0,\tilde{u}_0)\in\mathcal{C}_{j},$$
then the following holds:
\begin{enumerate}
	\item $(\lambda_{1,0},\lambda_{2,0}^*,\ldots,\lambda_{j,0}^*,\lambda_{j+1,0},\ldots,\lambda_{k,0},\vec{b}_0,\vec{x}_0,\tilde{u}_0)\in\mathcal{C}_{1}$, where
	{\begin{align*}
		&\lambda_{j,0}^*=F_j(\vec{b}_0,\vec{x}_0,\tilde{u}_0,\lambda_{1,0},\lambda_{j+1,0},\ldots,\lambda_{k,0})>0,\\
		&\lambda_{j-1,0}^*=F_{j-1}(\vec{b}_0,\vec{x}_0,\tilde{u}_0,\lambda_{1,0},\lambda_{j,0}^*,\lambda_{j+1,0},\ldots,\lambda_{k,0})>0,\\
		&\quad\vdots\\
		&\lambda_{2,0}^*=F_{2}(\vec{b}_0,\vec{x}_0,\tilde{u}_0,\lambda_{1,0},\lambda_{3,0}^*,\ldots,\lambda_{j,0}^*,\lambda_{j+1,0},\ldots,\lambda_{k,0})>0.
	\end{align*}}
	\item Let $u(t)$ be the solution of \eqref{CP} with initial data $$u_0=F(\lambda_{1,0},\lambda_{2,0}^*,\ldots,\lambda_{j,0}^*,\lambda_{j+1,0},\ldots,\lambda_{k,0},\vec{b}_0,\vec{x}_0,\tilde{u}_0),$$ and $\{\lambda_{i}(t)\}_{i=1}^k$ be the corresponding scaling parameters, then for all $t\in[0,T^*)$, $1\leq i_1,i_2\leq j$, we have
	$$\frac{1}{2^{k+1}}\leq \frac{\lambda_{i_1}(t)}{\lambda_{i_2}(t)}\leq 2^{k+1}.$$
\end{enumerate}

Now for $j=j_0$, we consider all $(\lambda_{1,0},\lambda_{j_0+1,0},\ldots,\lambda_{k,0},\vec{b}_0,\vec{x}_0,\tilde{u}_0)$, such that there exist $\lambda_{2,0},\ldots,\lambda_{j_0,0}>0$, such that
$$(\lambda_{1,0},\lambda_{2,0},\ldots,\lambda_{k,0},\vec{b}_0,\vec{x}_0,\tilde{u}_0)\in\mathcal{C}_{j_0}.$$

We similarly denote by $S_{j_0}$ the set of all $\lambda_{j_0,0}$ such that there exist $\lambda_{2,0},\ldots,\lambda_{j_0-1,0}>0$ such that
$$(\lambda_{1,0},\lambda_{2,0},\ldots,\lambda_{k,0},\vec{b}_0,\vec{x}_0,\tilde{u}_0)\in\mathcal{C}_{j_0-1}.$$
It is easy to see from the definition of $\mathcal{C}_{j_0-1}$ and $\mathcal{C}_{j_0}$, that $S_{j_0}$ is an interval and not empty. Moreover, from the induction hypothesis, for all $\lambda_{j_0,0}\in S_{j_0}$, we have:
$$(\lambda_{1,0},\lambda_{2,0}^*,\ldots,\lambda_{j_0-1,0}^*,\lambda_{j_0,0},\ldots,\lambda_{k,0},\vec{b}_0,\vec{x}_0,\tilde{u}_0)\in\mathcal{C}_1,$$
where
{\begin{align*}
&\lambda_{j_0-1,0}^*=F_{j_0-1}(\vec{b}_0,\vec{x}_0,\tilde{u}_0,\lambda_{1,0},\lambda_{j_0,0},\ldots,\lambda_{k,0})>0,\\
&\lambda_{j_0-2,0}^*=F_{j_0-2}(\vec{b}_0,\vec{x}_0,\tilde{u}_0,\lambda_{1,0},\lambda^*_{j_0-1,0},\lambda_{j_0,0},\ldots,\lambda_{k,0})>0,\\
&\quad\vdots\\
&\lambda_{2,0}^*=F_{2}(\vec{b}_0,\vec{x}_0,\tilde{u}_0,\lambda_{1,0},\lambda_{3,0}^*,\ldots,\lambda^*_{j_0-1,0},\lambda_{j_0,0},\ldots,\lambda_{k,0})>0.
\end{align*}}

Next we define $S^<_{j_0}$ be the set of all $\lambda_{j_0,0}\in S_{j_0}$ such that the solution $u(t)$ with initial data $$u_0=F(\lambda_{1,0},\lambda_{2,0}^*,\dots,\lambda_{j_0-1,0}^*,\lambda_{j_0,0},\ldots,\lambda_{k,0},\vec{b}_0,\vec{x}_0,\tilde{u}_0),$$ 
satisfies $\lambda_{j_0}(t_0)/\lambda_{i_0}(t_0)<1/2^{k+1}$, for some $t_0\in[0,T^*)$ and some $i_0\in\{1,\ldots,j_0-1\}$. 

Similarly, we define $S^>_{j_0}$ be the set of all $\lambda_{j_0,0}\in S_{j_0}$ such that the solution $u(t)$ with initial data $$u_0=F(\lambda_{1,0},\lambda_{2,0}^*,\dots,\lambda_{j_0-1,0}^*,\lambda_{j_0,0},\ldots,\lambda_{k,0},\vec{b}_0,\vec{x}_0,\tilde{u}_0),$$ satisfies $\lambda_{j_0}(t_0)/\lambda_{i_0}(t_0)>2^{k+1}$, for some $t_0\in[0,T^*)$ and some $i_0\in\{1,\ldots,j_0-1\}$.

We have the same observations:
\begin{enumerate}
	\item $S^<_{j_0}$ and $S^>_{j_0}$ are both contained in $S_{j_0}$ and open.
	\item $S^<_{j_0}\cap S^>_{j_0}$ is empty. Otherwise, there exist $\lambda_{j_0,0}\in S_{j_0}$, $i_1,i_2\in\{1,\ldots,j_0-1\}$ $t_0\in[0,T^*)$ such that $\lambda_{j_0}(t_0)/\lambda_{i_1}(t_0)>2^{k+1}$, $\lambda_{j_0}(t_0)/\lambda_{i_2}(t_0)<1/2^{k+1}$. Then we have $\lambda_{i_1}(t_0)/\lambda_{i_2}(t_0)<1/2^{2k+2}$, which is a contradiction due to the choice of $\lambda_{2,0}^*,\ldots,\lambda_{j_0-1,0}^*.$
	\item $S^<_{j_0}\not=S_{j_0}$, $S^>_{j_0}\not=S_{j_0}$. Suppose we have $S^<_{j_0}=S_{j_0}$. From our induction hypothesis, we know for all $i_1,i_2\in\{1,j_0+1,\ldots,k\}$, 
	$$\frac{1}{2^{k+1-j_0}}<\frac{\lambda_{i_1,0}}{\lambda_{i_2,0}}<2^{k+1-j_0}.$$
	
	Choose $\lambda_{j_0,0}>0$, such that
	$$\lambda_{j_0,0}=(2^{k+2-j_0}-\delta)\lambda_{i_0,0},$$
	where $i_0\in\{1,j_0+1,\ldots,k\}$, $\lambda_{i_0,0}=\min_{i\in\{1,j_0+1,\ldots,k\}}\lambda_{i,0}$, and $\delta>0$ is a small enough constant.
	Then for all $i_1,i_2\in\{1,j_0,\ldots,k\}$, we have
	$$\frac{1}{2^{k+1-(j_0-1)}}<\frac{\lambda_{i_1,0}}{\lambda_{i_2,0}}< 2^{k+1-(j_0-1)}.$$
	
	So there exist $\lambda_{2,0},\ldots,\lambda_{j_0-1,0}$ such that
	$$(\lambda_{1,0},\lambda_{2,0},\ldots,\lambda_{j_0-1,0},\lambda_{j_0,0},\ldots,\lambda_{k,0},\vec{b}_0,\vec{x}_0,\tilde{u}_0)\in\mathcal{C}_{j_0-1},$$
	or equivalently $\lambda_{j_0,0}\in S_{j_0}$($=S_{j_0}^<$). From our induction hypothesis, we know that
	$$(\lambda_{1,0},\lambda_{2,0}^*,\ldots,\lambda_{j_0-1,0}^*,\lambda_{j_0,0},\ldots,\lambda_{k,0},\vec{b}_0,\vec{x}_0,\tilde{u}_0)\in\mathcal{C}_1,$$
    where
    {\begin{align*}
    &\lambda_{j_0-1,0}^*=F_{j_0-1}(\vec{b}_0,\vec{x}_0,\tilde{u}_0,\lambda_{1,0},\lambda_{j_0,0},\ldots,\lambda_{k,0})>0,\\
    &\lambda_{j_0-2,0}^*=F_{j_0-2}(\vec{b}_0,\vec{x}_0,\tilde{u}_0,\lambda_{1,0},\lambda^*_{j_0-1,0},\lambda_{j_0,0},\ldots,\lambda_{k,0})>0,\\
    &\quad\vdots\\
    &\lambda_{2,0}^*=F_{2}(\vec{b}_0,\vec{x}_0,\tilde{u}_0,\lambda_{1,0},\lambda_{3,0}^*,\ldots,\lambda^*_{j_0-1,0},\lambda_{j_0,0},\ldots,\lambda_{k,0})>0.
    \end{align*}}
    
    But on the other hand, we have:
    $$\frac{\lambda_{j_0}}{\lambda_{1,0}}=\frac{\lambda_{j_0,0}}{\lambda_{i_0,0}}\times\frac{\lambda_{i_0,0}}{\lambda_{1,0}}>(2^{k+1-(j_0-1)}-\delta)2^{-k-1+j_0}>\frac{9}{5},$$
    if $\delta$ is small enough. From \eqref{22}, we know that 
    $$\frac{\lambda_{j_0}(0)}{\lambda_1(0)}>\frac{2}{3}\times\frac{\lambda_{j_0,0}}{\lambda_{1,0}}>\frac{6}{5}>\frac{10}{9},$$
    where $\{\lambda_{\ell}(t)\}_{\ell=1}^k$ are the scaling parameters of solution $u(t)$ with initial data
    $$u_0=F(\lambda_{1,0},\lambda_{2,0}^*,\dots,\lambda_{j_0-1,0}^*,\lambda_{j_0,0},\ldots,\lambda_{k,0},\vec{b}_0,\vec{x}_0,\tilde{u}_0).$$
    By Lemma \ref{Lemma21}, we reach a contradiction. The proof of $S^>_{j_0}\not=S_{j_0}$ is the same.
\end{enumerate}

Therefore, $S_{j_0}/(S^<_{j_0}\cup S^>_{j_0})$ is not empty. Then  we only need to choose 
$$F_{j_0}(\vec{b}_0,\vec{x}_0,\tilde{u}_0,\lambda_{1,0},\lambda_{j_0+1,0},\ldots,\lambda_{k,0})=\inf S_{j_0}/(S^<_{j_0}\cup S^>_{j_0}),$$
which ends the argument of the induction and concludes the proof of the Lemma.
\end{proof}

{Now we turn back to the proof of Proposition \ref{211}. We call parameters $(\vec{b}_0, \vec{x}_0, \tilde{u}_0, \lambda_{1,0})$ \emph{admissible} if and only if there exist $\lambda_{2,0},\ldots,\lambda_{k,0}>0$ such that
$$(\lambda_{1,0},\lambda_{2,0},\ldots,\lambda_{k,0},\vec{b}_0,\vec{x}_0,\tilde{u}_0)\in\mathcal{C}_{k}.$$

Recall that this means 
$$u_0=\sum_{i=1}^k\frac{1}{\lambda_{i,0}^{\frac{2}{p-1}}}Q_{b_{i,0}}\bigg(\frac{x-x_{i,0}}{\lambda_{i,0}}\bigg)+\tilde{u}_0,$$
satisfies \eqref{27}--\eqref{29} with \eqref{28} replaced by
$$\frac{1}{2}<\frac{\lambda_{i,0}}{\lambda_{j,0}}<2.$$

Then we define $\mathcal{O}_{k,p}\subset H^2$ as following:
\begin{equation*}
\mathcal{O}_{k,p}=\big\{F(\lambda_{1,0},\lambda_{2,0}^*,\ldots,\lambda_{k,0}^*,\vec{b}_0,\vec{x}_0,\tilde{u}_0)|(\vec{b}_0, \vec{x}_0, \tilde{u}_0, \lambda_{1,0})\text{ is admissible}\big\},
\end{equation*}
where 
\begin{align*}
&\lambda_{k,0}^*=F_k(\vec{b}_0,\vec{x}_0,\tilde{u}_0,\lambda_{1,0})>0,\\
&\lambda_{k-1,0}^*=F_{k-1}(\vec{b}_0,\vec{x}_0,\tilde{u}_0,\lambda_{1,0},\lambda_{k,0}^*)>0,\\
&\quad\vdots\\
&\lambda_{2,0}^*=F_{2}(\vec{b}_0,\vec{x}_0,\tilde{u}_0,\lambda_{1,0},\lambda_{3,0}^*,\ldots,\lambda_{k,0}^*)>0.
\end{align*}}
It is easy to see that $\mathcal{O}_{k,p}$ is an infinite set.

On the other hand, the choice of $\{\lambda_{j,0}^*\}_{j=2}^k$ implies that the scaling parameters of solution $u(t)$ with initial data in $\mathcal{O}_{k,p}$ satisfy:
$$\frac{1}{2^{k+1}}\leq \frac{\lambda_i(t)}{\lambda_j(t)}\leq2^{k+1},$$
for all $1\leq i,j\leq k$. This concludes the proof of Proposition \ref{211}.
\end{proof}

{\begin{remark}
From the construction of the subset $\mathcal{O}_{k,p}$, if one can show that the functions $F_j$ ($j=2,\ldots,k$) are actually in $C^1$, then the subset $\mathcal{O}_{k,p}$ has a codimension of $k-1$ in $H^2$. But this seems to be nontrivial.
\end{remark}

\begin{remark}
From the proof of Lemma \ref{212}, the choice of $\lambda^*_{j,0}$ ($j=2,\ldots,k$) may not be unique%
\footnote{The set $S_{j_0}/(S_{j_0}^<\cup S^>_{j_0})$ may contains more than one element.}%
. Here in Lemma \ref{212}, we basically chose the ``infimum" of all possible $\lambda^*_{j,0}$, which ensures that the functions $F_j$ are all continues. This argument is crucial to show that the blow-up points depend continuously on the initial data.
\end{remark}}

For this nonempty subset $\mathcal{O}_{k,p}$, the most important feature is that for $u_0\in\mathcal{Q}_{k,p}$, the estimates \eqref{BA1}--\eqref{BA7} can be improved on $[0,T^*)$. Hence from Remark \ref{rem1}, we have $T^*=T$. More precisely, we have: 
\begin{proposition}\label{prop}
If $u_0\in \mathcal{O}_{k,p}$, then the following estimates hold on $[0,T^*)$:
\begin{gather}
0<\lambda_j(t)<2,\label{BB1}\\
\frac{1}{2^{k+1}}\leq\frac{\lambda_i(t)}{\lambda_j(t)}\leq2^{k+1},\label{BB2}\\
|x_i(t)-x_j(t)|\geq b_c^{-75}, \text{ for all }i\not=j,\label{BB3}\\
|b_{j}(t)-b_c|\leq b_c^{\frac{3}{2}+2\nu},\label{BB4}\\
\|\varepsilon_j(t)\|_{L^{p_0}}\leq b_c^{\frac{13}{28}},\label{BB5}\\
\big\|\big(\varepsilon_j(t)\big)_y\big\|_{L^2}\leq b_c^{\frac{3}{4}},\label{BB6}\\
\mathcal{N}_j(t)\leq b_c^{3+8\nu},\label{BB7}
\end{gather}
\end{proposition}
{From a standard bootstrap argument, we know that $T^*=T$. i.e. the estimates \eqref{BB1}--\eqref{BB7} hold for all $t\in [0,T)$. }

{ Indeed, if $T^*<T$, then since \eqref{BB1}--\eqref{BB7} are strictly stronger than \eqref{BA1}--\eqref{BA7}, together with the continuity of the geometrical parameters and the the error term, we can see that the geometrical decomposition \eqref{GD} and \eqref{BA1}--\eqref{BA7} actually hold on $[0,T^*+\delta)$ for some $\delta>0$ small enough. This contradicts with the definition of $T^*$, since it is the maximal time that geometrical decomposition and \eqref{BA1}--\eqref{BA7} hold.
We will prove Proposition \ref{prop} in Section 4.1.

We know from Proposition \ref{prop} that \eqref{ME1}--\eqref{ME3} are approximations of \eqref{FODE}. Hence, after integrating (following from similar arguments as in \cite[Section 6]{L1}), they have similar behaviors. More precisely, we have:}
{\begin{proposition}\label{prop1}
For all $u_0\in\mathcal{O}_{k,p}$, we have:
\begin{enumerate}
\item Finite time blow-up with self-similar rate: We have $T<+\infty$, and for all $j=1,\ldots,k$ and $t\in[0,T)$:
\begin{equation}
\label{112}
3(1-\nu)b_c\leq \frac{\lambda_j^3(t)}{T-t}\leq 3(1+\nu)b_c.
\end{equation}
\item The translation parameters converge to pairwise distinct points: For all $j=1,\ldots,k$,
\begin{gather} 
x_j(t)\rightarrow x_j(T),\;\text{as }t\rightarrow T,\label{113}\\
|x_j(0)-x_j(T)|\lesssim \frac{1}{b_c},\label{1131}\\
x_i(T)\not=x_j(T)\;\text{for all }1\leq i\not=j\leq k.\label{1132}
\end{gather}
\item Convergence in subcritical Lebesgue spaces: for all $q\in[2,\frac{p-1}{2})$,
\begin{equation}\label{114}
u(t)\rightarrow u^*\; \text{in } L^q.
\end{equation}
\item For $R$ small enough, we have:
\begin{gather}
(1-\delta(p))\int\mathcal{Q}_p^2\leq \frac{1}{R^{2\sigma_c}}\int_{|x-x_j(T)|<R}|u^*|^2\leq (1+\delta(p))\int\mathcal{Q}_p^2,\label{115}\\
\limsup_{R\rightarrow0}\frac{1}{R^{2\sigma_c}}\int_{|x-z|<R}|u^*|^2\leq \delta(p),\quad\text{for all }z\notin\{x_1(T),\ldots,x_k(T)\},\label{116}
\end{gather}
which implies that the blow-up set of $u(t)$ is exactly $\{x_1(T),\ldots,x_k(T)\}$.
\item The map from $\mathcal{O}_{k,p}$ to $\mathbb{R}^k$:
\begin{equation}
\label{106}
u_0\mapsto (x_1(T),\ldots,x_k(T))
\end{equation}\
is continuous under the topology of $H^1$ and $\mathbb{R}^k$.
\end{enumerate}
\end{proposition}

\begin{remark}
Proposition \ref{prop1} implies that for all $u_0\in\mathcal{O}_{k,p}$, the corresponding solution will blow up in finite time with self-similar rate, and has exactly $k$ blow-up points.
\end{remark}}

\section{Monotonicity tools and estimates on the error term}
{ In this section, we will derive some crucial estimates on the error tern $\varepsilon_j$, which imply the bootstrap bounds \eqref{BB6} and \eqref{BB7} immediately. Such estimates are similar to \cite[Lemma 4.1, Proposition 5.2]{L1}, and are the continuation of the monotonicity formula developed in \cite{MM3} and \cite{MMR1}.}
\subsection{Monotonicity of the energy.}{In this subsection we will give a control of $\|(\varepsilon_j)_y\|_{L^2}$ and $\|(\varepsilon_j)_y\|_{L^2(b_c^{-20}>y>\kappa B)}$, which implies the bootstrap bound \eqref{BB6}. These estimates provide a good control of the $L^{\infty}$ norm of $\varepsilon_j$ on the right.}
\begin{proposition}
	For all $j=1,\ldots,k$ the following estimates hold for all $s_j\in[0,s_j^*)$:
	\begin{gather}
	\int \big(\varepsilon_j(s_j)\big)_y^2 \lesssim b_c^{\frac{3}{2}+\frac{\nu}{2}}\label{31},\\
	\int_{\kappa B<y<b_c^{-20}} \big(\varepsilon_j(s_j)\big)_y^2 \lesssim b_c^{\frac{55}{7}}\label{32}.
	\end{gather}
\end{proposition}
\begin{proof}
	The proof of \eqref{31} is a consequence of the energy conservation law. Indeed, we have:
	\begin{multline}
	2\lambda_j(s_j)^{2(1-\sigma_c)}E(u_0)=2E(\mathfrak{Q}_j)+\sum_{i=1}^k\bigg(\frac{\lambda_j(s_j)}{\lambda_i(s_j)}\bigg)^{2(1-\sigma_c)}\int(\varepsilon_i)_y(Q_{b_i})_y\\+\int(\varepsilon_j)_y^2-\frac{2}{p+1}\int\big{(}(\mathfrak{Q}_j+\varepsilon_j)^{p+1}-\mathfrak{Q}_j^{p+1}\big{)},\\
	\end{multline}
	where 
	$$\mathfrak{Q}_j(s_j,y)=\sum_{i=1}^k\bigg(\frac{\lambda_j(s_j)}{\lambda_i(s_j)}\bigg)^{\frac{2}{p-1}}Q_{b_i}\bigg(\frac{\lambda_j(s_j)y+x_j(s_j)-x_i(s_j)}{\lambda_i(s_j)}\bigg).$$
	
	For the terms appear in the above summation, their supports are pairwise disjoint. So we have:
	\begin{align*}
	\big|E\big(\mathfrak{Q}_j(s_j,\cdot)\big)\big|&=\Bigg|\sum_{i=1}^kE\Bigg(\bigg(\frac{\lambda_j(s_j)}{\lambda_i(s_j)}\bigg)^{\frac{2}{p-1}}Q_{b_i}\bigg(\frac{\lambda_j(s_j)\cdot+x_j(s_j)-x_i(s_j)}{\lambda_i(s_j)}\bigg)\Bigg)\Bigg|\\
	&\leq \sum_{i=1}^k6^{k+1}(|b_i-b_c|+b_c^3)\leq 10^kb_c^{\frac{3}{2}+\nu}\leq b_c^{\frac{3}{2}+\frac{\nu}{2}}.
	\end{align*}
Here we use the fact that {$10^k\leq b_c^{-\nu/4}$}. The rest terms can be estimated similarly like what we do in \cite{L1}%
\footnote{See details in the first part of Section 4 in \cite{L1}.}%
, thus we conclude the proof of \eqref{31}.

The proof of \eqref{32} is quite different from the single blow-up point case. We first choose 2 smooth functions $\theta$ and $\eta$, such that $\theta>0$, $\theta(y)=e^{-|y|}$ for $|y|>1$ and $\eta(y)=1$ for $y<1$, $\eta(y)=e^{-y}$ for $y>2$.

We introduce the following notations:
$$\Theta(y)=\frac{\int_{-\infty}^{y}\theta(y')\,dy'}{\int_{-\infty}^{+\infty}\theta(y')\,dy'},\quad\tilde{y}=\frac{y-\kappa B}{\sqrt{B}},$$
$$\Psi_B(y)=\Theta(\tilde{y})\eta(b_c^{20}y).$$

Then we assume that for all $t\in[0,T^*)$, $j\in\{1,\ldots,k\}$, 
\begin{equation}\label{33}
\int_{y>\kappa B}\big[\varepsilon_j(t,y)\big]^2\Psi_B(y)\,dy\leq b_c^{\frac{15}{2}}.
\end{equation}
Since this estimate is satisfied for $t=0$, so we only need to improve this estimate to:
\begin{equation}\label{34}
\int_{y>\kappa B}\big[\varepsilon_j(t,y)\big]^2\Psi_B(y)\,dy\lesssim b_c^{\frac{55}{7}}.
\end{equation}

To do this, we fix $t\in[0,T^*)$. For $\tau\in[0,t]$, we introduce the localized energy:
$$\widetilde{E}(\tau)=\int\Big{(}\frac{1}{2}|u_x(\tau)|^2-\frac{1}{p+1}|u(\tau)|^{p+1}\Big{)}\Psi_B\bigg(\frac{x-x_j(\tau)}{\lambda_j(\tau)}\bigg)\,dx.$$
A direct computation shows:
\begin{equation}\label{36}\begin{split}
&\quad\lambda_j(t)^{2(1-\sigma_c)}\widetilde{E}(t)\\
&=\frac{1}{2}\int\big{(}(Q_{b_j})_y+(\varepsilon_j)_y\big{)}^2\Psi_B(y)\,dy-\frac{1}{p+1}\int|Q_{b_j}+\varepsilon_j|^{p+1}\Psi_B(y)\,dy+O(b_c^{20})\\
&\gtrsim \int_{y>\kappa B}(\varepsilon_j)_y^2(t)\Psi_B-e^{-\frac{\kappa \sqrt{B}}{2}}\int_{y<\frac{\kappa B}{2}}\Big{(}|(Q_{b_j})_y|^2+|Q_{b_j}|^{p+1}\Big{)}\\
&\quad-\int_{y>\frac{\kappa B}{2}}\Big{(}|(Q_{b_j})_y|^2+|Q_{b_j}|^{p+1}\Big{)}-\int_{y>\frac{\kappa B}{2}}|\varepsilon|^{p+1}\Psi_B-e^{-\frac{\kappa \sqrt{B}}{2}}\int_{y<\frac{\kappa B}{2}}|\varepsilon|^{p+1}-b_c^{20},
\end{split}\end{equation}
where we use the fact that $\Psi_B(y)\leq e^{-\kappa \sqrt{B}/4}$ if $y<\kappa B/2$.

Next, we have:
$$\int_{y>\frac{\kappa B}{2}}|\varepsilon|^{p+1}\Psi_B\leq \bigg\|\varepsilon_j(t,\cdot)^{\frac{p_0}{2}+1}\eta(b_c^{20}\cdot)^{\frac{p_0+2}{2(p+1-p_0)}}\bigg\|^{\frac{2(p+1-p_0)}{p_0+2}}_{L^{\infty}(y>\kappa B)}\int \varepsilon_j^{p_0}.$$
For $y>\kappa B$, we have the following estimate:
\begin{align*}
&\quad\;\bigg||\varepsilon_j(t,y)|^{\frac{p_0}{2}+1}\eta(b_c^{20}y)^{\frac{p_0+2}{2(p+1-p_0)}}\bigg|\leq\bigg||\varepsilon_j(t,y)|^{\frac{p_0}{2}+1}\eta(b_c^{20}y)^{\frac{1}{2}}\bigg|\\
&\leq \bigg|\int_y^{+\infty}(\varepsilon_j)_y|\varepsilon_j|^{\frac{p_0}{2}}(t,y')\eta(b_c^{20}y')^{\frac{1}{2}}\,d y'\bigg|+\bigg|b_c^{20}\int_{y}^{+\infty}|\varepsilon_j(t,y')|^{\frac{p_0}{2}+1}\frac{\eta'(b_c^{20}y')}{\eta(b_c^{20}y')^{\frac{1}{2}}}\,dy'\bigg|\\
&\lesssim \|\varepsilon_j\|_{L^{p_0}}^{\frac{p_0}{2}}\bigg(\int_{y'>\kappa B}(\varepsilon_j)_y^2\eta(b_c^{20}y')\,dy'\bigg)^{\frac{1}{2}}+b_c^{10}\|\varepsilon_j\|_{L^{p_0+2}}^{\frac{p_0+2}{2}}\bigg(\int\frac{\big(\eta'(y')\big)^2}{\eta(y')}\,dy'\bigg)^{\frac{1}{2}}\\
&\lesssim b_c^{\frac{173}{40}},
\end{align*}
where we use \eqref{BA5}, \eqref{210}, \eqref{33} and the basic fact that $\eta(b_c^{20}y)\leq2\Psi_B(y)$ for $y>\kappa B$. Combining the above 2 estimates, we have (recall that $p_0=\frac{5}{2}$ and $p$ is slightly larger than $5$): 
\begin{equation}\label{37}
\int_{y>\kappa B}|\varepsilon_j|^{p+1}\Psi_B\lesssim b_c^{\frac{173p-156}{90}}\leq b_c^{\frac{55}{7}}
\end{equation}
On the other hand, from Sobolev embedding and \eqref{BA7}, we have:
\begin{equation}\label{38}
\int_{\kappa B/2<y<\kappa B}|\varepsilon_j|^{p+1}\Psi_B\leq \|\varepsilon\|^{p+1}_{H^1(\kappa B/2<y<\kappa B)}\leq b_c^{9}.
\end{equation}

Injecting \eqref{37} and \eqref{38} into \eqref{36}, we have:
$$\int_{\kappa B<y<b_c^{-20}}\varepsilon_j(t)^2\lesssim b_c^{\frac{55}{7}}+\lambda_j(t)^{2(1-\sigma_c)}\widetilde{E}(t)$$

Now it remains to estimate $\lambda_j(t)^{2(1-\sigma_c)}\widetilde{E}(t)$. To do this, we first use the Kato's localized identities for the energy to compute the derivative of $\widetilde{E}(\tau)$:
\begin{equation}\label{39}\begin{split}
\frac{d}{d{\tau}}\widetilde{E}(\tau)=&-\frac{1}{2}\int(u_{xx}+u|u|^{p-1})^2g_x-\int u_{xx}^2g_x\\
&+p\int u|u|^{p-2}u_x^2g_x+\frac{1}{2}\int u_x^2g_{xxx}\\
&-\frac{x_t(\tau)}{\lambda(\tau)}\int\Big{(}\frac{1}{2}|u_x(\tau)|^2-\frac{1}{p+1}|u(\tau)|^{p+1}\Big{)}\Psi_B'\bigg{(}\frac{x-x_j(\tau)}{\lambda_j(\tau)}\bigg{)}\,dx\\
&-\frac{\lambda_t(\tau)}{\lambda(\tau)}\int\Big{(}\frac{1}{2}|u_x(\tau)|^2-\frac{1}{p+1}|u(\tau)|^{p+1}\Big{)}\bigg{(}\frac{x-x(\tau)}{\lambda(\tau)}\bigg{)}\Psi_B'\bigg{(}\frac{x-x_j(\tau)}{\lambda_j(\tau)}\bigg{)}\,dx\\
\end{split}\end{equation}
where
$$g(x,\tau)=\Psi_B\bigg(\frac{x-x_j(\tau)}{\lambda_j(\tau)}\bigg).$$

We claim there exists a universal constant $C$ such that for all $\tau\in[0,t]$,
\begin{equation}\label{310}
\frac{d}{d{\tau}}\widetilde{E}(\tau)\leq \frac{Cb_c^{\frac{62}{7}}}{\lambda(\tau)^{3+2(1-\sigma_c)}}
\end{equation}
We denote by
\begin{align*}
&g_1(x,\tau)=\frac{1}{\sqrt{B}\lambda_j(\tau)}\theta\bigg(\frac{x-x_j(\tau)}{\sqrt{B}\lambda_j(\tau)}-\kappa B\bigg)\eta\bigg(b_c^{20}\frac{x-x_j(\tau)}{\lambda_j(\tau)}\bigg),\\
&g_2(x,\tau)=\frac{b_c^{20}}{\lambda_j(\tau)}\Theta\bigg(\frac{x-x_j(\tau)}{\sqrt{B}\lambda_j(\tau)}-\kappa B\bigg)\eta'\bigg(b_c^{20}\frac{x-x_j(\tau)}{\lambda_j(\tau)}\bigg),\\
&g_{3}(x,\tau)=\frac{1}{\big(\sqrt{B}\lambda_j(\tau)\big)^3}\theta''\bigg(\frac{x-x_j(\tau)}{\sqrt{B}\lambda_j(\tau)}-\kappa B\bigg)\eta\bigg(b_c^{20}\frac{x-x_j(\tau)}{\lambda_j(\tau)}\bigg).
\end{align*}
Then we have
$$g=g_1+g_2,\quad g_{xxx}-g_3=O\bigg(\frac{b_c^{20}}{\big(\sqrt{B}\lambda_j(\tau)\big)^3}\bigg).$$
So we can rewrite \eqref{39} as following:
$$\frac{d}{d{\tau}}\widetilde{E}(\tau)=I+II+III+IV+V,$$
where 
\begin{align*}
&I=-\frac{1}{2}\int(u_{xx}+u|u|^{p-1})^2g_2-\int u_{xx}^2g_2,\\
&II=p\int u|u|^{p-2}u_x^2g_2+\frac{1}{2}\int u_x^2(g_{xxx}-g_3),\\
&III=-x_t(\tau)\int\Big{(}\frac{1}{2}|u_x(\tau)|^2-\frac{1}{p+1}|u(\tau)|^{p+1}\Big{)}g_2(x,\tau)\,dx,\\
&IV=-\lambda_t(\tau)\int\Big{(}\frac{1}{2}|u_x(\tau)|^2-\frac{1}{p+1}|u(\tau)|^{p+1}\Big{)}\bigg{(}\frac{x-x(\tau)}{\lambda(\tau)}\bigg{)}g_2(x,\tau)\,dx,\\
&V=-\frac{1}{2}\int(u_{xx}+u|u|^{p-1})^2g_1-\int u_{xx}^2g_1\\
&\qquad+p\int u|u|^{p-2}u_x^2g_1+\frac{1}{2}\int u_x^2g_{3}\\
&\qquad-x_t(\tau)\int\Big{(}\frac{1}{2}|u_x(\tau)|^2-\frac{1}{p+1}|u(\tau)|^{p+1}\Big{)}g_1(x,\tau),dx\\
&\qquad-\lambda_t(\tau)\int\Big{(}\frac{1}{2}|u_x(\tau)|^2-\frac{1}{p+1}|u(\tau)|^{p+1}\Big{)}\bigg{(}\frac{x-x(\tau)}{\lambda(\tau)}\bigg{)}g_1(x,\tau)\,dx.\\
\end{align*}
It is easy to estimate $V$ by following the same argument as in \cite[Section 4]{L1}. From the properties of $\eta$ (i.e. exponential decay on the right), we know that on the support of $g_2$ and $g_{xxx}-g_3$, the following term is negligible:
$$\sum_{i=1}^k\frac{1}{\lambda_i(\tau)^{\frac{2}{p-1}}}Q_{b_i(\tau)}\bigg(\frac{x-x_i(\tau)}{\lambda_i(\tau)}\bigg).$$
Together with \eqref{BA6}, \eqref{BA7}, \eqref{210} and \eqref{ME2} we have:
\begin{align*}
&\big|II\big|\lesssim\frac{b_c^{20}}{\lambda_j(\tau)^{3+2(1-\sigma_c)}}\bigg(b_c^{100}+\int\big(|\varepsilon_j|^{p-1}(\varepsilon_j)_y^2+(\varepsilon_j)_y^2\big)\bigg)\leq \frac{Cb_c^{\frac{62}{7}}}{\lambda_j(\tau)^{3+2(1-\sigma_c)}},\\
&\big|III\big|\lesssim \frac{b_c^{20}x_{s_j}}{\lambda_j(\tau)^{4+2(1-\sigma_c)}}\bigg[b_c^{100}+\int\big(\varepsilon_j^2+|\varepsilon_j|^{p+1}\big)\bigg]\leq \frac{Cb_c^{\frac{62}{7}}}{\lambda_j(\tau)^{3+2(1-\sigma_c)}}.
\end{align*}
While for $IV$ we have $g_2\leq0$, $\lambda_t\leq 0$, and on the support of $g_2$, $(x-x_j(\tau))/\lambda_j(\tau)\geq0$. So we have:
$$-\lambda_t(\tau)\int\Big{(}\frac{1}{2}|u_x(\tau)|^2\Big{)}\bigg{(}\frac{x-x(\tau)}{\lambda(\tau)}\bigg{)}g_2(x,\tau)\,dx\leq 0.$$
Moreover, from \eqref{ME1} and the choice of $\eta$ we have:
\begin{align*}
&\quad\;\bigg|\lambda_t(\tau)\int\frac{1}{p+1}|u(\tau)|^{p+1}\bigg{(}\frac{x-x(\tau)}{\lambda(\tau)}\bigg{)}g_2(x,\tau)\,dx\bigg|\\
&\lesssim\frac{b_c}{\lambda_j(\tau)^{3+2(1-\sigma_c)}}\int|\varepsilon(\tau,y')|^{p+1}(b_c^{20}y')\eta'(b_c^{20}y')\,dy'\\
&\lesssim \frac{b_c}{\lambda_j(\tau)^{3+2(1-\sigma_c)}}\int_{y'>\kappa B}|\varepsilon(\tau,y')|^{p+1}\eta(b_c^{20}y')^{\frac{99}{100}}\,dy'\\
&\leq \frac{Cb_c^{\frac{62}{7}}}{\lambda_j(\tau)^{3+2(1-\sigma_c)}},
\end{align*}
where the last inequality follows from the same argument which is used to estimate \eqref{37}.
Thus we obtain:
$$IV\leq \frac{Cb_c^{\frac{62}{7}}}{\lambda_j(\tau)^{3+2(1-\sigma_c)}}.$$

Finally, we deal with $I$. First of all, we have
\begin{equation}\label{311}
I\lesssim \frac{b_c^{20}}{\lambda_j(\tau)}\bigg(\int (u_{xx}^2+|u|^{2p})\bigg)\lesssim \frac{b_c^{20}}{\lambda_j(\tau)}\int u_{xx}^2+\frac{b_c^{10}}{\lambda_j(\tau)^{3+2(1-\sigma_c)}},
\end{equation}
where we use the fact that
\begin{align*}
\int|u|^{2p}&\lesssim \frac{1}{\lambda_j(\tau)^{2+2(1-\sigma_c)}}\int\bigg(|\varepsilon_j(\tau)|^{2p}+\sum_{i=1}^k\bigg(\frac{\lambda_j(\tau)}{\lambda_i(\tau)}\bigg)^{\frac{3p+1}{p-1}}|Q_{b_i(\tau)}|^{2p}\bigg)\\
&\lesssim \frac{10^k}{\lambda_{j}(\tau)^{2+2(1-\sigma_c)}}\leq \frac{b_c^{-\nu/4}}{\lambda_{j}(\tau)^{2+2(1-\sigma_c)}}.
\end{align*}
While for $\int u_{xx}^2$, we can use pseudo-conservation law to estimate. Precisely, we have the following estimate for all $\tau_0\in[0,\tau]$:
\begin{equation}\label{312}
\frac{d}{d\tau_0}E_2(\tau_0)\lesssim \int u_x^3|u|^{2p-3}(\tau_0)+\int u_x^5|u|^{p-4}(\tau_0),
\end{equation}
where
$$E_2(\tau_0)=\int u_{xx}^2(\tau_0)-\frac{5p}{3}\int u_x^2|u|^{p-1}(\tau_0).$$
It is easy to prove \eqref{312} by integrating by parts.

Now we assume the following a $priori$ estimate for all $\tau_0\in[0,\tau]$:
\begin{equation}\label{313}
\int u_{xx}(\tau_0)^2\leq \frac{b_c^{-8}}{\lambda_j(\tau_0)^{2+2(1-\sigma_c)}}.
\end{equation}
Then Sobolev embedding implies that:
$$\|u_x(\tau_0)\|_{L^{\infty}}\leq \|u_x(\tau_0)\|_{L^2}^{\frac{1}{2}}\|u_{xx}(\tau_0)\|_{L^2}^{\frac{1}{2}}\lesssim \frac{b_c^{-2-\nu}}{\lambda_j(\tau_0)^{\frac{1}{2}+2(1-\sigma_c)}},$$
where we use the fact that
$$\int u_x^2(\tau_0)\lesssim \frac{b_c^{-\nu/4}}{\lambda_j(\tau_0)^{2(1-\sigma_c)}}.$$

From \eqref{312} and \eqref{313} we obtain
\begin{align*}
\frac{d}{d\tau_0}E_2(\tau_0)&\lesssim \|u_x\|_{L^{\infty}}\|u\|_{L^{\infty}}^{2p-3}\int u_x^2+\|u_x\|^3_{L^{\infty}}\|u\|_{L^{\infty}}^{p-4}\int u_x^2\\
&\lesssim \frac{b_c^{-6-10\nu}}{\lambda_j(\tau_0)^{5+2(1-\sigma_c)}}.
\end{align*}
Note that for $\beta>3$,
$$\int_0^{\tau_0}\frac{1}{\lambda_j(\tau')^{\beta}}\,d\tau'\leq \frac{2}{b_c(\beta-3)\lambda_j(\tau_0)^{\beta-3}}.$$
We have:
$$E_2(\tau_0)\lesssim E_2(0)+\frac{b_c^{-7-10\nu}}{\lambda_j(\tau_0)^{2+2(1-\sigma_c)}}.$$
The conditions on the initial data lead to:
\begin{gather*}
E_2(0)\lesssim \frac{b_c^{-10\nu}}{\lambda_j(0)^{2+2(1-\sigma_c)}}\leq \frac{b_c^{-10\nu}}{\lambda_j(\tau_0)^{2+2(1-\sigma_c)}}.
\end{gather*}
Together with
$$\int u_x^2|u|^{p-1}(\tau_0)\lesssim \frac{b_c^{-10\nu}}{\lambda_j(\tau_0)^{2+2(1-\sigma_c)}},$$
we have for all $\tau_0\in[0,\tau]$,
$$\int u_{xx}(\tau_0)^2\leq \frac{b_c^{-7-10\nu}}{\lambda_j(\tau_0)^{2+2(1-\sigma_c)}}.$$
From a standard bootstrap argument (if $\nu$ is small enough), we have shown that
$$\int u_{xx}(\tau)^2\leq \frac{b_c^{-7-10\nu}}{\lambda_j(\tau)^{2+2(1-\sigma_c)}}.$$

Injecting this into \eqref{311} we get
$$I\leq \frac{b_c^{10}}{\lambda_j(\tau)^{3+2(1-\sigma_c)}},$$
which concludes the proof of \eqref{310}, hence the proof of \eqref{34} and \eqref{32}. 
\end{proof}

The main goal of this subsection is to control the $L^{\infty}$ norm of $\varepsilon_j$ on the right, i.e.
\begin{corollary}
The following $L^{\infty}$ control hold for all $t\in[0,T^*)$:	
	\begin{equation}\label{314}
	\|\varepsilon_j(t)\|_{L^{\infty}(b_c^{-10}>y>\kappa B)}\lesssim b_c^{2}.
	\end{equation}
\end{corollary}
\begin{proof}
	Let $\eta_0$ be a smooth function with $\eta_0(y)=1$ for $y<1$, $\eta_0(y)=0$ for $y>2$. Let $f(y)=\varepsilon_j(y)\eta_0(b_c^{10}y)$.
	Applying the localized Gagliardo-Nirenberg inequality to $f$, we have
	\begin{align*}
	\|f\|_{L^{\infty}(y>\kappa B)}&\leq \|f\|_{L^{p_0}}^{\frac{p_0}{p_0+2}}\|f_y\|_{L^2(y>\kappa B)}^{\frac{2}{p_0+2}}\\
	&\lesssim  \|\varepsilon_j\|_{L^{p_0}}^{\frac{p_0}{p_0+2}}\bigg(\big\|(\varepsilon_j)_y\big\|_{L^2(b_c^{20}>y>\kappa B)}^{\frac{2}{p_0+2}}+\Big(b_c^5\|\varepsilon_j\|_{L^{\infty}}\Big)^{\frac{2}{p_0+2}}\bigg)\\
	&\lesssim b_c^2.
	\end{align*}
\end{proof}

\subsection{Monotonicity formula}
{ In this subsection, we will derive a monotonicity formula for the localized Sobolev norm of $\varepsilon_j$, which will imply the bootstrap bound \eqref{BB7} immediately and is important in the derivation of the asymptotic dynamics of the flow. This formula here is almost the same to the one in \cite{L1}. Such monotonicity tools were introduced originally in \cite{MM3} and \cite{MMR1} for critical gKdV.}

Recall from (2.32), the definition of $\varphi$. We let $\psi$, $\eta_0$ be 2 other smooth functions such that:
\begin{align}
&\psi(y)=\begin{cases}
e^{y} & \text{ for } y<-1,\\
1 & \text{ for } y>-\kappa,
\end{cases}\quad \psi'\geq0,\\
&\eta_0(y)=\begin{cases}
1 & \text{ for } y<1,\\
0 & \text{ for } y>2,
\end{cases}\quad \eta_0'\leq0.
\end{align}
Here, we observe that $\psi(-\kappa)=\varphi(-\kappa)+\kappa$, and $\psi(y)=\varphi(y)$ for all $y<-1$, so we may assume in addition:
\begin{equation}\label{APOW}
\varphi(y)\leq \psi(y)\leq (1+3\kappa)\varphi(y), \text{ for all }y\leq-\kappa.
\end{equation}
\begin{remark}
	It is easy to check that for every $\frac{1}{2}>\kappa>0$, such $\psi$ and $\varphi$ exist.
\end{remark}
Now, recall $B=b_c^{-\frac{1}{20}}$. We let
$$\psi_B(y)=\psi(\frac{y}{B})\eta_0(b_c^{10}y),\quad \zeta_B(y)=\varphi_B\eta_0(\frac{y}{B^2}).$$
and then define the following $Lyapunov$ functional for $\varepsilon_j$:
\begin{equation}
\mathcal{F}_j=\int\bigg{[}(\varepsilon_j)_y^2\psi_B+(\varepsilon_j)^2\zeta_B-\frac{2}{p+1}\big{(}|\varepsilon_j+Q_{b_j}|^{p+1}-Q_{b_j}^{p+1}-(p+1)\varepsilon_j Q_{b_j}^p\big{)}\psi_B\bigg{]}.
\end{equation}
Our main goal here is the following monotonicity formula for $\mathcal{F}_j$:
\begin{proposition}[The second monotonicity formula]\label{MF}
	There exists a universal constant $\mu>0$ and $0<\kappa<\frac{1}{2}$, such that for all $j=1,\ldots,k$, $ s_j\in[0,s^*_j)$, the following holds:
	\begin{enumerate}
		\item Lyapunov control:
		\begin{equation}\label{MF1}
		\frac{d}{ds_j}\mathcal{F}_j+\mu\int\big{(}(\varepsilon_j)_y^2+(\varepsilon_j)^2\big{)}\varphi'_B\lesssim b_c^{\frac{7}{2}};
		\end{equation}
		\item Coercivity of $\mathcal{F}$: there exists a universal constant $\kappa>0$ such that
		\begin{equation}\label{MF2}
		\mathcal{N}_j-b_c^{\frac{7}{2}}\lesssim\mathcal{F}_j\lesssim \mathcal{N}_j+b_c^{\frac{7}{2}}.
		\end{equation}
	\end{enumerate}
\end{proposition}
\begin{remark}
The bootstrap bound \eqref{BB7} follows immediately from \eqref{MF1}, \eqref{MF2} and Grownwell's inequality.
\end{remark}
\begin{proof}
	The proof of Proposition \ref{MF} is exactly the same to the one in \cite[Section 5]{L1}. The idea is that the error term $\varepsilon_j$ has been ``localized" to the support of $Q_{b_j}$, due to our choice of the weight functions. Then the estimate is exactly the same to the single blow-up point case.
	
	The only difference here is that we add a cut-off on the right of $\psi_B$. This will lead to some additional terms on the right hand side of \eqref{MF1}. But if we check the proof of \cite[Proposition 5.1]{L1}, we will see these additional terms can always be bounded by
	$$-b_c^{10}\int(\varepsilon_j)^2_y(s_j,y')\eta_0'(b_c^{10}y')\,dy',$$
	and hence bounded by $b_c^{10}$. 
\end{proof}

\section{Existence of blow-up solutions with exactly $k$ blow-up points}
This section is devoted to prove Proposition \ref{prop} and Proposition \ref{prop1}. Hence for all $u_0\in\mathcal{O}_{k,p}$, the corresponding solution has exactly $k$ blow-up points.
\subsection{Closing the bootstrap}
In this subsection we will finish the bootstrap argument and finally prove Proposition \ref{prop}.
	
\begin{proof}[Proof of Proposition \ref{prop}]
The bounds \eqref{BB1}--\eqref{BB4} are consequences of the modulation estimates \eqref{ME1}--\eqref{ME3}. Indeed, \eqref{BB1} follows from the fact that $0<\lambda_j(0)<2$ and that $\lambda_j(t)$ is decreasing. \eqref{BB2} is just the definition of $\mathcal{O}_{k,p}$. For \eqref{BB3}, we have
\begin{align*}
|x_i(t)-x_j(t)|&\geq |x_i(0)-x_j(0)|-\int_0^t|(x_i)_t|+|(x_j)_t|\\
&\geq b_c^{-80}-2\int_0^t\bigg(\frac{1}{\lambda_i^2}+\frac{1}{\lambda_j^2}\bigg)\\
&\geq b_c^{-80}+\frac{1}{b_c}\int_0^t(\lambda_i+\lambda_j)_t\geq b_c^{-75}.
\end{align*}
	
While for \eqref{BB4}, suppose for some $s_{j,0}\in(0,s^*_j)$, we have $b(s_{j,0})> b_c+ b_c^{\frac{3}{2}+2\nu}$. By the choice of the initial data, i.e. \eqref{225}, we can find some $s_{j,1}\in[0,s_{j,0})$ such that $b(s_{j,1})=b_c+b_c^{\frac{3}{2}+\frac{5}{2}\nu}$ and $b(s_j)\geq b_c+b_c^{\frac{3}{2}+\frac{5}{2}\nu}$ for all $s_j\in[s_{j,1},s_{j,0})$. Then $b_{s_j}(s_{j,1})\geq 0$. From \eqref{BA7} and \eqref{ME3}, we have:
\begin{equation}
b_{s_j}(s_{j,1})\leq-c_p\big(b(s_{j,1})-b_c\big)b_c+b_c^{\frac{5}{2}+3\nu}= -c_p b_c^{\frac{5}{2}+\frac{5\nu}{2}}+b_c^{\frac{5}{2}+3\nu}<0,
\end{equation}
if $b_c$ is small enough such that $b_c^{\nu}\ll 1$. We get a contradiction. The opposite bound is similar.
		
The proof of \eqref{BB5}--\eqref{BB7} is parallel to the one for the single blow-up point case in \cite{L1}, using the similar monotonicity tools developed in the previous section.

{Indeed, we can see \eqref{BB6} is a direct consequence of \eqref{31}, and \eqref{BB7} follows from \eqref{MF1}, \eqref{MF2} and Grownwell's inequality. Finally, similar as \cite[Section 6.1]{L1}, we can prove \eqref{BB5} by a refined Strichartz estimates introduced in \cite{Fos}.}
\end{proof}	

\subsection{Proof of Proposition \ref{prop1}} First we prove the finite time blow-up and self-similar result, i.e. \eqref{112}.
From Proposition \ref{prop} and \eqref{ME1}, we know for all $t\in[0,T)$, $j\in\{1,\ldots,k\}$, 
\begin{equation}\label{41}
(1-\nu)b_c\leq -(\lambda_j)_t\lambda_j^2\leq(1+\nu)b_c.
\end{equation}
Integrating \eqref{41} from $0$ to $t$ we know that for all $j\in\{1,\ldots,k\}$,
\begin{equation}
\label{411}
\forall t\in[0,T),\quad(1-\nu)b_ct\leq\frac{1}{3}\lambda_j^3(0)\text{ and hence }T\leq\frac{\lambda_j^3(0)}{3b_c(1-\nu)}<+\infty.
\end{equation}
So the solution blows up in finite time. Then from local Cauchy theory, we know $\lim_{t\rightarrow T}\|u_x(t)\|_{L^2}=+\infty$. But we know from the geometrical decomposition:
$$\|u_x(t)\|_{L^2}\sim \sum_{i=1}^k\frac{1}{\lambda_i(t)^{1-\sigma_c}}.$$
Combining with \eqref{BB2}, we have for all $j\in\{1,\ldots,k\}$,
$$\lim_{t\rightarrow T}\lambda_j(t)=0.$$
We then integrate \eqref{41} from $t$ to $T$  to obtain:
$$\forall t\in[0,T),\quad (1-\nu)b_c(T-t)\leq\frac{\lambda_j^3(t)}{3}\leq(1+\nu)b_c(T-t),$$
which implies \eqref{112}.

{ While for \eqref{1131}, from \eqref{ME2}, \eqref{21} and \eqref{411}, we have for all $j=1,\ldots,k$,
$$|x_j(0)-x_j(T)|\lesssim \int_0^{+\infty}|(x_j)_{s_j}|\,ds_j\lesssim \int_0^{T}\frac{dt}{\lambda_j(t)}\lesssim \int_0^T\frac{dt}{\sqrt[3]{b_c(T-t)}}\lesssim \frac{T^{2/3}}{b^{1/3}_c}\lesssim \frac{1}{b_c}.$$}
The proof of \eqref{1132}--\eqref{116} is exactly the same as the one in \cite[Section 6.2]{L1}. 

Finally, for \eqref{106}, from \eqref{continuity} and Lemma \ref{212}, we only need to show that the blow-up time is continuous with respect to the initial data in $\mathcal{O}_{k,p}.$  More precisely, we have:
\begin{lemma}
	Consider $\{u_{0,n}\}_{n=1}^{+\infty}\subset \mathcal{O}_{k,p}$, $u_0\in\mathcal{O}_{k,p}$, such that $u_{0,n}$ converges to $u_0$ in $H^1$ as $n\rightarrow+\infty$. Let $u_n(t)$, $u(t)$ be the corresponding solutions to \eqref{CP}, and $T_n$, $T$ be the corresponding blow-up times, then we have $$\lim_{n\rightarrow+\infty}T_n=T.$$
\end{lemma}
\begin{proof}First of all, from a classical argument of continuity with respect to the initial data (i.e. the perturbation theory%
\footnote{See for example Lemma 2.4 in \cite{MMR2}.}%
), we obtain:
\begin{equation}
\label{510}
\liminf_{n\rightarrow+\infty}T_n\geq T.
\end{equation}

On the other hand, for all $\delta>0$, there exists $n(\delta)>0$, such that if $n>n(\delta)$, $u_n(T-\delta)$ exists. Integrating \eqref{41} from $T-\delta$ to $T_n$, we have
$$T_n<T-\delta+\frac{\lambda^3_{j,n}(T-\delta)}{3(1-\nu)b_c},$$
where $\lambda_{j,n}(t)$ is the $j$-th scaling parameter of the solution $u_n(t)$. Let $n\rightarrow+\infty$, we will obtain:
$$\limsup_{n\rightarrow +\infty}T_n\leq T-\delta+\frac{\lambda^3_j(T-\delta)}{3(1-\nu)b_c}\leq T+2\delta.$$
Then let $\delta\rightarrow0$, we have $\limsup_{n\rightarrow+\infty}T_n\leq T$, which concludes the proof of the lemma.
\end{proof}

Therefore, we finish the proof of \eqref{106} and hence the proof of Proposition \ref{prop1}.

\section{Proof of Theorem \ref{MT} by Brouwer's theorem}
In this section, we will prove Theorem \ref{MT}. Actually, Proposition \ref{prop1} has already given the existence of solutions with exactly $k$ blow-up points. And here we will use another topological argument to show that the blow-up points can be chosen arbitrarily.

Given any $k$ points $\{x_1,\ldots,x_k\}$, we want to find a solution whose blow-up set is exactly $\{x_1,\ldots,x_k\}$. 

{\noindent{\bf Step 1.} First, we show that if
\begin{equation}
\label{LD}
|x_i-x_j|\geq b_c^{-120},\text{ for all }i\not=j,
\end{equation}
then there exists a solution $u(t)$ satisfying \eqref{10001} and \eqref{10002}, whose blow-up set is exactly $\{x_1,\ldots,x_k)$.}

For $j=1,\ldots,k$, we let $I_j=[x_j-b_c^{-3},x_j+b_c^{-3}].$ Then for all $x_{i,0}\in I_i$, we have
$$|x_{i_1,0}-x_{i_2,0}|\geq b_c^{-100},\text{ for all }i_1\not=i_2.$$

{Next, we fix suitable $\lambda_{1,0}$, $b_{1,0},\ldots,b_{k,0}>0$ and $\tilde{u}_0\in C^{\infty}_0$,} such that for all 
$$(x_{1,0},\ldots,x_{k,0})\in I_1\times\cdots\times I_k,$$
there exist $(\lambda_{2,0},\ldots,\lambda_{k,0})$ such that conditions \eqref{27}--\eqref{29} is satisfied for%
\footnote{We mention here that this is possible due to the assumption \eqref{LD}.}
$$v_0(x)=F(\vec{\lambda_0},\vec{b}_0,\vec{x}_0,\tilde{u}_0)=\sum_{i=1}^k\frac{1}{\lambda_{i,0}^{\frac{2}{p-1}}}Q_{b_{i,0}}\bigg(\frac{x-x_{i,0}}{\lambda_{i,0}}\bigg){}+\tilde{u}_0(x).$$

Then, for all $(x_{1,0},\ldots,x_{k,0})\in I_1\times\cdots\times I_k$, we can consider the solution $u(t)$ with initial data
\begin{align*}
u_0=F(\lambda_{1,0},\lambda_{2,0}^*,\ldots,\lambda_{k,0}^*,\vec{b}_0,\vec{x}_0,\tilde{u}_0),
\end{align*}
where%
\footnote{Here, $F_j$, $j=2,\ldots,k$ are the continuous functions defined in Lemma \ref{212}.}
\begin{align*}
&\lambda_{k,0}^*=F_k(\vec{b}_0,\vec{x}_0,\tilde{u}_0,\lambda_{1,0})>0,\\
&\lambda_{i,0}^*=F_i(\vec{b}_0,\vec{x}_0,\tilde{u}_0,\lambda_{1,0},\lambda_{i+1,0}^*,\ldots,\lambda_{k,0}^*)>0,\text{ for }i=2,\ldots,k-1.
\end{align*}
{Obviously, we have $u_0\in C^{\infty}_0.$}

{Then, from Proposition \ref{prop}, we know that the geometrical decomposition \eqref{GD} and the estimates \eqref{BB1}--\eqref{BB7} hold for all $t\in[0,T)$, where $T$ is the maximal life span of $u(t)$. From Proposition \eqref{prop1}, we know that $u(t)$ blows up in finite time (i.e. $T<+\infty$), and has exactly $k$ blow-up points, i.e. $\{x_1(T),\ldots, x_k(T)\}$. It is easy to check that $u(t)$ satisfies \eqref{10001}, \eqref{10002} and \eqref{10003}, and the blow-up set of $u(t)$ is $\{x_1(T),\ldots,x_k(T)\}$ (the limit of the translation parameters $x_j(t)$ as $t\rightarrow T$). Next, we  define a map $M$ from $D=I_1\times\cdots\times I_k$ to $\mathbb{R}^k$ as following:
$$M=M_2\circ M_1,$$
where
\begin{gather*}
M_1:\; D\rightarrow \mathcal{O}_{k,p},\\
(x_{1,0},\ldots,x_{k,0})\mapsto u_0=F(\lambda_{1,0},\lambda_{2,0}^*,\ldots,\lambda_{k,0}^*,\vec{b}_0,\vec{x}_0,\tilde{u}_0),
\end{gather*}
and 
\begin{gather*}
M_2:\; \mathcal{O}_{k,p}\rightarrow \mathbb{R}^k,\\
u_0\mapsto \{x_1(T),\ldots,x_k(T)\}\text{ (the blow-up set of $u(t)$.)}
\end{gather*}
From \eqref{106} and the fact that $F_j$ is continuous for all $j=2,\ldots,k$, it is easy to see that the maps $M_1$ and $M_2$ are continuous. Hence $M$ is continuous. }

Now we claim there exists a $(x_{1,0},\ldots,x_{k,0})\in D$ such that 
\begin{equation}\label{53}
M(x_{1,0},\ldots,x_{k,0})=(x_{1},\ldots,x_{k}).
\end{equation}

{From the construction of geometrical decomposition (i.e. the argument of implicit function theorem), we have
$$|x_{i,0}-x_i(0)|\ll1,$$
for all $i=1,\ldots,k$. Together with \eqref{1131}, we have for all $i\in\{1,\ldots,k\}$,
\begin{equation}\label{52}
|x_{i,0}-x_i(T)|\leq b_c^{-2}.
\end{equation}

We then introduce the following topological lemma, which is a corollary of the \emph{Brouwer's fixed point theorem}, \cite{Bro}.  

\begin{lemma}\label{TL}
Let $f$ be a continuous map from $\mathbb{R}^k$ to $\mathbb{R}^k$, and $T_r=[-r,r]^k\subset \mathbb{R}^k$ be a cube centered at $0$, for some $r>0$. Suppose we have for all $y\in\partial T_r$,
\begin{equation}\label{51}
|f(y)-y|<r,
\end{equation}
then there exists a $y_0\in T_r$ such that $f(y_0)=0$.
\end{lemma}

\begin{proof}[Proof of Lemma \ref{TL}]
Suppose for all $y\in T_r$, $f(y)\not=0$. Then we can define a map $g$ from $T_r$ to $\partial T_r$ as following:
$$g(y)=\partial T_r\cap \{tf(y)|t\geq0\}.$$
It is easy to check that $g$ is well-defined and continuous. The assumption \eqref{51} ensures that for all $y\in \partial T_r$, and $t\in[0,1]$, we have
$$tg(y)+(1-t)y\not=0,$$
which implies that $g|_{\partial T_r}$ is homotopic to ${\rm Id}_{\partial T_r}$. Indeed we can consider the following map:
\begin{gather*}
G:\; [0,1]\times \partial T_r\rightarrow \partial T_r,\\
(t,x)\mapsto \partial T_r\cap \{s[tg(y)+(1-t)y]|s\geq0\}.
\end{gather*} 
It is easy to check that $G$ is well-defined and continuous. Moreover, we have $G(0,y)=y$, $G(1,y)=g(y)$, for all $y\in\partial T_r$. Thus $g|_{\partial T_r}$ is homotopic to ${\rm Id}_{\partial T_r}$.

Then the homeomorphism of the homology groups induced by $g$ (i.e. $g_*$: $H_*(T_r)\rightarrow H_*(\partial T_r)$) is surjective. But this is a contradiction, since $H_{k-1}(T_r)=0$, $H_{k-1}(\partial T_r)=\mathbb{Z}$. Therefore, we conclude the proof of Lemma \ref{TL}.
\end{proof}

Now we apply Lemma \ref{TL} to $f=M$, and $T_r=D$ with $r=b_c^{-3}$. From \eqref{52}, we can see that condition \eqref{51} is satisfied. Then we obtain \eqref{53}, which concludes the proof of Theorem \ref{MT} under the assumption of \eqref{LD}.\\
 
\noindent{\bf Step 2.} Now for arbitrarily $k$ pairwise distinct points $\{x_1,\ldots,x_k\}$, choose $\underline{\lambda}>0$, such that
$$\min_{1\leq i\not=j\leq k}|\underline{\lambda}x_i-\underline{\lambda}x_j|\geq b_c^{-120}.$$

Now from the above arguments, there exists a solution $v(t)$ blowing up in finite time $T_v<+\infty$, whose blow-up set is $\{\underline{\lambda}x_1,\ldots,\underline{\lambda}x_k\}$. Moreover, for $t$ close to $T_v$, there exist $\lambda_{j,v}(t)$ and $\tilde{v}(t,x)$ such that
\begin{gather*}
v(t,x)=\sum_{j=1}^k\frac{1}{\lambda^{\frac{2}{p-1}}_{j,v}(t)}\mathcal{Q}_p\bigg(\frac{x-\underline{\lambda}x_j}{\lambda_{j,v}(t)}\bigg)+\tilde{v}(t,x),\\
\frac{\lambda^3_{j,v}(t)}{T_v-t}\sim 3b_c, \quad \lambda_{j,v}(t)^{1-\sigma_c}\|\tilde{v}_x(t)\|_{L^2}\leq \delta(p).
\end{gather*}

Then we let
$$u(t,x)=\underline{\lambda}^{\frac{2}{p-1}}v(\underline{\lambda}^3t,\underline{\lambda}x).$$
It is easy to see from Remark \ref{symmetry} that $u(t)$ is a solution to \eqref{CP} blowing up in finite time $T_u=\underline{\lambda}^{-3}T_v<+\infty$. And its blow-up set is exactly $\{x_1,\ldots,x_k\}$. Moreover, $u(t)$ satisfies \eqref{10001}, \eqref{10002} and \eqref{10003} with
$$\lambda_j(t)=\frac{\lambda_{j,v}(\underline{\lambda}^3t)}{\underline{\lambda}},\quad\tilde{u}(t,x)=\underline{\lambda}^{\frac{2}{p-1}}\tilde{v}(\underline{\lambda}^3t,\underline{\lambda}x),\quad \text{for all $t$ close to }T_u=\underline{\lambda}^{-3}T_v.$$
Therefore, we conclude the proof of Theorem \ref{MT}.}

\bibliographystyle{amsplain}
\bibliography{ref}

\end{document}